\documentclass[review]{elsarticle}
\usepackage{lineno,hyperref}
\usepackage{amsmath,amssymb,amsthm}
\usepackage{graphicx}
\usepackage{color}
\usepackage{appendix}
\usepackage{titlesec}
\usepackage{mathrsfs}
\newtheorem{theorem}{Theorem}[section]
\newtheorem{lemma}[theorem]{Lemma}
\newenvironment{proofof}[1]{%
  \renewcommand{\proofname}{Proof of #1}%
  \begin{proof}%
}{%
  \end{proof}%
}
\newenvironment{1proofof}[1]{%
  \renewcommand{\proofname}{I. Proof of #1}%
  \begin{proof}%
}{%
  \end{proof}%
}
\newenvironment{2proofof}[1]{%
  \renewcommand{\proofname}{II. Proof of #1}%
  \begin{proof}%
}{%
  \end{proof}%
}
\usepackage[perpage]{footmisc}
\newtheorem{remark}[theorem]{Remark}
\newtheorem*{assumption}{(H)}

\newtheorem{corollary}[theorem]{Corollary}
 \numberwithin{equation}{section}

\usepackage{extarrows}
\journal{arXiv}

\begin{document}

	\begin{frontmatter}
		\title{Rapid boundary stabilization of 1D nonlinear parabolic equations
		\tnoteref{mytitlenote}}
		\tnotetext[mytitlenote]{This work was partially supported by the National Natural Science Foundation of China under the grant 12422118.}
		
		\author[mysecondaryaddress]{Yu Xiao}
		\ead{xiaoyu\_math@whu.edu.cn}
	    \medskip
	      \author[my4address]{Can Zhang\corref{mycorrespondingauthor}}
		\cortext[mycorrespondingauthor]{Corresponding author}
		\ead{canzhang@whu.edu.cn}
		\address[mysecondaryaddress]{School of Mathematics and Statistics, Wuhan University, Wuhan 430072, P.R.China}
		\address[my4address]{School of Mathematics and Statistics, Wuhan University, Wuhan 430072, P.R.China}

        \begin{abstract}
	   In this paper, we focus on the rapid boundary stabilization of 1D nonlinear parabolic equations via the modal decomposition method. The nonlinear term is assumed to satisfy certain local Lipschitz continuity and global growth conditions. Through the
       modal decomposition, we construct a feedback control that modifies only the unstable eigenvalues to achieve spectral reduction. Under this control, we establish locally rapid stabilization by estimating the nonlinearity in Lyapunov stability analysis. Furthermore, utilizing the dissipative property, we derive a globally rapid stabilization result for dissipative systems such as the Burgers equation and the Allen-Cahn equation.
        \end{abstract}


    \end{frontmatter}

\section{Introduction}
\subsection{Motivation}
Stability analysis has long been a central topic in control theory, particularly the problem of feedback stabilization. By applying feedback control on the boundary or over a subset of the domain, where the control is selected based on current or historical state information, we expect the closed-loop system to exhibit stability in its large time behavior. A variety of approaches have been developed to address the feedback stabilization problem for partial differential equations (PDEs), such as the backstepping method (e.g., \cite{CN} for the heat equation and \cite{CX} for the Burgers equation), the modal decomposition method (see \cite{KF,SWZ,X} for the heat equation, \cite{L1} for the Kuramoto–Sivashinsky equation, \cite{LP2,PWF} for the Burgers equation, \cite{X1} for the Navier–Stokes equation, \cite{LHM} for the Fisher equation), as well as other methods (e.g., \cite{BW,TWX} based on LQ theory).

As one of the most widely studied classes of PDEs, linear parabolic equations have seen significant progress in feedback stabilization(\cite{CN,X}). However, numerous physical problems involve nonlinearities, and some progress has been made in recent years on the feedback stabilization of nonlinear parabolic equations (e.g., \cite{B, KF2,KF3,LHM,LP2,LP3,PWF}). In particular, the modal decomposition method serves as an important tool. Through spectral analysis and state decomposition, this method uses a finite number of modes to construct a feedback control. Its success lies in the fact that such a control modifies only the unstable eigenvalues, thereby achieving stabilization.

When addressing the controllability of semilinear parabolic equations, the authors of \cite{CT1} combined the modal decomposition method with a constructive Lyapunov function to establish stability over a finite time interval. Based on this stability, they further derived global controllability by using existing results on approximate controllability. In \cite{B,LHM,M2}, relying on a fixed-point theorem rather than Lyapunov analysis, the authors proved the locally rapid stabilization of multi-dimensional semilinear parabolic equations.

Recently, for the case where the nonlinear term is globally Lipschitz continuous, the authors of \cite{KF2,LP3} established the well-posedness of the closed-loop system through the analyticity of the Sturm–Liouville operator. Simultaneously, by utilizing the Lipschitz continuity, they effectively used the norm of the state to bound the nonlinearity. Then, through a Lyapunov analysis similar to that in \cite{CT1}, they demonstrated a globally rapid stabilization result related to the Lipschitz coefficient. Regarding the semilinear equations in \cite{KF3}, the authors also achieved globally rapid stabilization. Taking into account specific nonlinearities, the authors of \cite{LP2,PWF} established the locally rapid stabilization of the Burgers equation. More precisely, they demonstrated that the nonlinear term has a negligible effect in the Lyapunov analysis when the initial value is sufficiently small, and thus attained the exponential stability inherited from the linearized equation. Motivated by these works, estimating the effect of the nonlinear term is of crucial importance in the Lyapunov stability analysis of nonlinear parabolic equations.

The aim of this paper is to construct a boundary control for 1D nonlinear parabolic equations using the modal decomposition method, so as to achieve well-posedness and rapid stabilization.

\subsection{Main result}
Consider the nonlinear parabolic equation with a boundary control
\begin{equation}\label{Dirichlet-sys}
\left\{\begin{array}{ll}
 \partial_t y(t,x) =\partial_x (a(x)\partial_x y( t,x))+{b(x)}y( t,x) +qy( t,x)+f(y( t,x))&\text{for }(t,x)\in \mathbb{R}^+ \times (0,1),\\[2mm]
y(t,0)=u(t)\quad \text{and}\quad y(t,1)=0&\text{for }t\in \mathbb{R}^+,\\[2mm]
y(0,x)=y_0(x )&\text{for }x\in (0,1),
\end{array}\right.
 \end{equation}
where $a\in C^2([0,1])$ with {$\min_{x\in[0,1]}a(x)>0$}, ${b}\in C([0,1])$ with {$\max_{x\in[0,1]}b(x)< 0$}, $q\in \mathbb{R}$, $u(\cdot)$ is the control, and the nonlinear term $f$ satisfies that $f(0)=0$.


Throughout the paper, we suppose the following assumption.
\begin{assumption}
\label{a1}
 The nonlinear term $f$ is assumed to satisfy the following two conditions:
\begin{itemize}
    \item[(i)] For each $y_0 \in H^1(0,1)$ and $r>0$, there exists a constant $C$ depending on $\|y_0\|_{H^1(0,1)}$ and $r$, such that  $$\|f(y_1)-f(y_2)\|_{L^2(0,1)}\le C(\|y_0\|_{H^1(0,1)},r) \|y_1-y_2\|_{H^1(0,1)},\quad \forall y_1,y_2\in B_{{H^1}}(y_0,r),$$
    where $B_{{H^1}}(y_0,r)=\left\{ y\in H^1(0,1):\|y-y_0\|_{H^1(0,1)}\le r\right\} $.
    \item[(ii)] There exists a strictly increasing function  
$c\in C([0,+\infty))$  satisfying $c(0) = 0$, such that 
$$  
\|f(y)\|_{L^2(0,1)} \leq c(\|y\|_{L^\infty(0,1)}) \|y\|_{H^1(0,1)},\quad \forall y \in H^1(0,1).  
$$  
\end{itemize}
    
\end{assumption}


The main result of this paper can be stated as follows.
\begin{theorem}
     \label{main-th1}
      Assume that $f$ satisfies the assumption (H). Given any $\delta>0$, 
      there are constants $\rho>0$, $k\in \mathbb{R}$ and $M\ge 1$, as well as a {feedback operator $\mathcal{K}:L^2(0,1)\to \mathbb{R}$}, all depending only on $\delta$, such that for every $y_0\in H^1_0(0,1)$ satisfying
         \begin{equation*}
             \|y_0\|_{H^1_0(0,1)}\le \rho,
         \end{equation*}
        the system \eqref{Dirichlet-sys} with {\begin{equation}
            \label{U-MAIN}u(t)=\int_0^t e^{k(t-s)}\mathcal{K}(y(s))ds
        \end{equation}}admits a unique solution $$y\in C([0,+\infty);H^1(0,1))\cap 
 C^1(\mathbb{R}^+;L^2(0,1))  $$ verifying the following exponential stability
         \begin{equation}
         \label{exp-s}
             \|y(t)\|_{H^1(0,1)}\le Me^{-\delta t}\|y_0\|_{H^1_0(0,1)},\;\;\forall t>0.
         \end{equation}
 \end{theorem}
 
Several remarks are given in order.

\begin{remark}
   As the exponential decay rate $\delta$ can be chosen arbitrarily, the resulting stabilization is called rapid stabilization (see \cite{LWXY,TWX} for the characterizations of rapid stabilization in linear control systems).
\end{remark}
\begin{remark}
   The control form \eqref{U-MAIN} is not the instantaneous feedback of the common type, but a feedback with memory, where the memory kernel $e^{k (t-\cdot)}$ acts on $\mathcal{K}(y(\cdot))$ defined in \eqref{feedback-k} below. In fact, if we regard $(y,u)$ as the state of the system \eqref{Dirichlet-sys}, then the original system is a PDE-ODE coupled system, i.e., we have an additional equation
    $$
    \frac{d}{dt}u(t)=ku(t)+\mathcal{K}(y(t))=\left(\mathcal{K},k\right)\begin{pmatrix}
        y(t)\\ u(t) \end{pmatrix},\quad \forall t>0.
    $$
    This implies that our control strategy takes an instantaneous feedback form with respect to $(y(t),u(t))$.
\end{remark}
\begin{remark}
The assumption (H) includes local Lipschitz continuity and global growth conditions. Among these, the Lipschitz continuity condition ensures the well-posedness of the closed-loop system, while the growth condition guarantees the estimability of the nonlinearity in Lyapunov analysis.
\end{remark}

The rest of the paper is organized as follows. Section \ref{sec2} presents some preliminaries. Section \ref{sec3} gives the feedback operator and provides the proof of {Theorem} \ref{main-th1}. Finally, as applications of {Theorem} \ref{main-th1}, {Section \ref{sec5} achieves globally rapid stabilization for the Burgers equation, as well as the Allen-Cahn equation.}

\section{Preliminaries}\label{sec2}
\subsection{Notation}
Given $p\ge 1$ and $m\ge 1$, $L^p(0,1)$ and $H^m(0,1)$ are the standard Lebesgue and Sobolev spaces of functions $f:(0,1)\to \mathbb{R}$ endowed with the norms $\|\cdot\|_{L^p(0,1)}$ and $\|\cdot\|_{H^m(0,1)}$. $H_0^{m}(0,1)$ is the closure in $H^{m}(0,1)$ of $C^{\infty}$ functions with compact support in $[0,1]$.
 
 When  $X_1$ and $X_2$ are two Hilbert spaces, we write
    $\mathcal{L}(X_1,X_2)$ as the space of all linear and bounded operators mapping from $X_1$ to $X_2$, equipped with the norm $\|\cdot\|_{\mathcal{L}}$, and further
    write  $\mathcal{L}(X_1):=\mathcal{L}(X_1,X_1)$. 
    The Euclidean norm on $\mathbb{R}^n$ is denoted by $|\cdot|_2$. We use $\|\cdot\|_2$ to represent the matrix norm induced by $|\cdot|_2$. For a symmetric matrix $P\in\mathbb{R}^{N\times N}$ or a self-adjoint operator $P:D(P)\to X$, the notation $ P\succeq 0$ (resp. $P\succ0$) means that $P$ is positive semi-definite (resp. positive definite). The eigenvalue of a matrix $P$ with the maximum (resp. minimum) real part is denoted by $\sigma_{\max}(P)$ (resp. $\sigma_{\min}(P)$).
   
   Given a Hilbert space $X$, let {$f\in C^1(\mathbb{R}^+,X)$, its derivative is denoted as $\dot{f}(t):=\frac{df(t)}{dt}$}. For two positive constants $F_1,F_2$, we say $F_1\lesssim F_2$ if there exists a positive constant $C$ independent of $F_i$ ($i=1,2$), such that $F_1\le CF_2$. Moreover, if $F_1\lesssim F_2$ and $F_2 \lesssim F_1$, we denote $F_1\sim F_2$. For convenience, we write $C\left(\cdots\right)$ for a positive constant that depends on what are included in the brackets. The same can be said about $C_1\left(\cdots\right),M_1\left(\cdots\right),\sigma_1\left(\cdots\right)$ and so on.

\subsection{Sturm-Liouville operator}
	
Let us introduce the Sturm-Liouville operator $A$ defined by
\begin{equation*}
\label{Sturm-Liouville-1}
       Af=-(\partial_x (a(x)\partial_x f)+{b(x)}f), \quad D( A)=H^2(0,1)\cap H^1_0(0,1).
\end{equation*}
       
\begin{lemma}
\label{A-property}
    The following statements about $A$ are valid (\cite[Chapter 2 and Chapter 3]{TW}).
\begin{enumerate}
\item[(i)]
${A}$ is self-adjoint and generates an analytic $C_0$-semigroup $\{e^{{A}t}\}_{t\ge 0}$ on $L^2(0,1)$. Its spectrum consists of a countable set of simple and real eigenvalues $\{\lambda_n\}_{n\geq 1}$ with $\lambda_n = \mathcal{O}(n^2)$. The corresponding eigenfunctions $\{\varphi_n\}_{n\ge 1}$ form an orthonormal basis in $L^2(0,1)$. 
        \item[(ii)] $A$ is diagonalizable and 
        \begin{equation*}\label{express3}
            \begin{aligned}
                Af=\sum_{n\ge 1}\lambda_n\left\langle f,\varphi_n \right\rangle\varphi_n,\;\; \forall f \in D(A).
            \end{aligned}
        \end{equation*}

\item[(iii)] ${A}$ admits a square root ${A}^{\frac{1}{2}}$ given by
\begin{equation*}\label{express4}
                {A}^{\frac{1}{2}}f=\sum_{n\ge 1}\lambda_n^{\frac{1}{2}}\left\langle f,\varphi_n \right\rangle\varphi_n,\;\; \forall  f \in D({A}^{\frac{1}{2}})=H^1_0(0,1).
        \end{equation*}
        {Moreover, there are two constants  $k_1\le 1\le k_2$ such that 
        \begin{equation*}
        k_1\|f\|_{H^1_0(0,1)} \le \|f\|_{D(A^{\frac{1}{2}})} \le k_2\|f\|_{H^1_0(0,1)},\;\; \forall f\in D(A^{\frac{1}{2}}).
        \end{equation*}
    }
\end{enumerate}
\end{lemma}

{Denoting by $X=L^2(0,1)$, we recall that the operator $A$ has a unique extension $\tilde{A}\in \mathcal{L}(X,D(A^*)')$, where $D(A^*)'$ is the dual of $D(A^*)$ with respect to the pivot space $X$. Here and throughout the paper, we still use $A$ to denote the extended operator $\tilde{A}$.}
    \subsection{Sobolev inequalities}
    We introduce some Sobolev inequalities, which will be used frequently later  (see \cite[Chapter 13]{Brezis}).
    \begin{lemma}[Poincar\'{e} inequality]
\label{PO}
There exists a constant $c_1>0$ such that
    \begin{equation}
        \label{po}
c_1\|y\|^2_{L^2(0,1)} \le \|\partial_x y\|^2_{L^2(0,1)},\;\; \forall y\in H^1_0(0,1).
\end{equation}
\end{lemma}
\begin{lemma}[Gagliardo-Nirenberg interpolation inequality]
\label{g-n}
Assume that $k, j \in \mathbb{N}$ with $j<k$. Let $\theta>0$ and $p>0$ satisfy that
$$ j / k \leq \theta \leq 1\quad \text{and}\quad
\frac{1}{p}-j=\frac{1}{2}-\theta k.
$$
Then there exists a constant $G=G\left(k,\theta\right)$ such that
$$
\|\partial_x^j y\|_{L^p(0,1)} \leq G\left\|\partial_x^k y\right\|_{L^2(0,1)}^\theta\|y\|_{L^2(0,1)}^{1-\theta}+G\|y\|_{L^2(0,1)},\;\; \forall y \in H^{k}(0,1).
$$
\end{lemma}
\begin{remark}
    We give some examples that are frequently used in applications.
    \begin{enumerate}
        \item If $p=+\infty$,  $k=1$ and $j=0$, then there is a constant $\tilde{c}_2>1$ so that
        \begin{equation}
        \label{Agmon}
            \| y\|_{L^{\infty}(0,1)} \leq \tilde{c}_2\left\| y\right\|_{H^{1}(0,1)},\;\; \forall y\in H^1(0,1).
        \end{equation}
        This, combined with Poincar\'{e} inequality \eqref{po}, implies that 
        \begin{equation}
        \label{Agmon1}
            \| y\|_{L^{\infty}(0,1)} \leq c_2\left\|\partial_x y\right\|_{L^{2}(0,1)}^{\frac{1}{2}}\|y\|_{L^{2}(0,1)}^{\frac{1}{2}},\;\; \forall y\in H_0^1(0,1),
        \end{equation}
        where the constant $c_2>1$ is independent of $y$. 
         \item If $p=2$, $k=2$ and $j=1$, then there exists a constant $c_3>1$ such that
        \begin{equation}
        \label{interpolation}
            \|\partial_x y\|_{L^{2}(0,1)} \leq c_3\left\| y\right\|_{H^{2}(0,1)}^{\frac{1}{2}}\|y\|_{L^{2}(0,1)}^{\frac{1}{2}},\;\;\forall y\in H^2(0,1).
        \end{equation}
    \end{enumerate}
\end{remark}
\section{Proof of Theorem \ref{main-th1}}\label{sec3}
This section is divided into three parts. We begin with the design of the feedback control. Then, we verify the well-posedness of the closed-loop system. Finally, as the main part of the proof, we focus on the rapid stabilization of \eqref{Dirichlet-sys}.
\subsection{Control design}
 
 For any $u\in \mathbb{R}$, we define $\mathcal{D}u=(1-x)u$ for $x\in(0,1)$\footnote{Actually $\mathcal{D} \in \mathcal{L}(\mathbb{R},H^1(0,1))$ with $\|\mathcal{D}\|_{\mathcal{L}}=\frac{3+\sqrt{3}}{3}>1$.}. Supposing $u\in C^1(\mathbb{R}^+)$ with $u(0)=0$, we denote $w(t)=y(t)-\mathcal{D} u(t)$ for $ t\ge 0$. Through the change of variables, it holds that
\begin{equation*}
    \begin{aligned}
\dot{w}(t)+{A} w (t)& =\dot{y}(t)-\mathcal{D} \dot{u}(t)+A y(t)-A \mathcal{D}u(t) \\
& =q w(t)+q\mathcal{D}u(t)-A \mathcal{D}u(t)-\mathcal{D} \dot{u}(t)+f\left(\mathcal{D}u(t)+w(t)\right), \;\;\forall t>0.
\end{aligned}
\end{equation*}

Taking arbitrarily $N\in \mathbb{N}$, we denote $w_n(t)=\langle w(t),\varphi_n\rangle$ and $ Y_N(t)=
(u(t),w_1(t) ,\cdots ,w_N(t))\in \mathbb{R}^{N+1}$. Consider the following \textbf{control design:}
\begin{equation}
    \label{controller}
    \dot{u}(t)=KY_N(t), \quad \forall t>0, 
\end{equation}
\text{where $K=[k_0,k_1,\cdots,k_N]\in \mathbb{R}^{1\times (N+1)}$}. Then the controlled system reads as
\begin{equation}
    \label{controlled-eqs}
    \left\{\begin{array}{ll}
 \dot{w}(t)+A w (t)=q w(t)+q\mathcal{D}u(t)-A \mathcal{D}u(t)-\mathcal{D} (KY_N(t))+f\left(\mathcal{D}u(t)+w(t)\right)&\text{ for } t\in \mathbb{R}^+,\\[2mm]
\dot{u}(t)=KY_N(t)&\text{ for } t\in  \mathbb{R}^+,\\[2mm]
w(0 )=y_0\quad \text{and} \quad u(0)=0.
\end{array}\right.
\end{equation}

\begin{lemma}\label{lemma-mode}
    If the system \eqref{controlled-eqs} is well-posed, i.e., there exists $T>0$ such that the system \eqref{controlled-eqs} has a unique classical solution\footnotemark[2] \footnotetext[2]{For the definition of the classical solutions, we refer to \cite[Chapter 4]{Pazy}.}, then for every $n\ge 1$, $w_n(t)$ satisfies the following equation
    \begin{equation}
        \label{mode-eq}  \dot{w}_n(t)+\lambda_n w_n(t)=qw_n(t)+a_n u(t) +b_n KY_N(t)+f_n(t),\;\;\forall t\in (0,T),
    \end{equation}
   where $a_n:=\langle q\mathcal{D}(1)-A \mathcal{D}(1),\varphi_n\rangle,b_n:= -\langle  \mathcal{D} (1),\varphi_n\rangle$ and $f_n(t):=\langle f\left(\mathcal{D}u(t)+w(t)\right),\varphi_n\rangle$. Moreover, $Y_N(t)$ satisfies the following equation
    \begin{equation}
    \label{mode-eq-new}
    \dot{Y}_N(t)=(A_0+B_0K)Y_N(t)+R_N(t),\;\;\forall t\in (0,T),
\end{equation}
where $A_0:=\left(\begin{array}{ccccc}
0 & 0 &0& \cdots &0 \\
a_1 &q-\lambda_1&0& \cdots & 0 \\
\vdots  & \vdots &\vdots& & \vdots\\
a_N&0&0&\cdots&q-\lambda_N
\end{array}\right)$, $B_0:=\begin{pmatrix}
1\\
b_1\\
\vdots \\
b_N
\end{pmatrix}$ and $R_N(t):=\begin{pmatrix}
0\\
 f_1(t)\\
\vdots \\
 f_N(t)
\end{pmatrix}$.
\end{lemma}
\begin{proof}
Taking the inner product of the first equation of \eqref{controlled-eqs} with respect to $\varphi_n$, one can readily verify \eqref{mode-eq}.
For $Y_N(t)$, we have
\begin{equation*}
\begin{aligned}
    \dot{Y}_N(t)&=\left(\begin{array}{ccccc}
k_0 & k_1 &k_2& \cdots &k_N \\
a_1+k_0b_1 &q-\lambda_1+k_1b_1&k_2b_1& \cdots & k_Nb_1 \\
\vdots  & \vdots &\vdots& & \vdots\\
a_N+k_0b_N&k_1b_N&k_2b_N&\cdots&q-\lambda_N+k_Nb_N
\end{array}\right){Y}_N(t)+\begin{pmatrix}
0\\
 f_1(t)\\
\vdots \\
 f_N(t)
\end{pmatrix}\\[2mm]
&=(A_0+B_0K)Y_N(t)+R_N(t),\;\;\forall t\in (0,T).
\end{aligned}
\end{equation*}
Hence, \eqref{mode-eq-new} follows.
\end{proof}

\begin{lemma}
  The pair $ (A_0,B_0)$ satisfies the Kalman rank condition\footnotemark[3] \footnotetext[3]{ For the definition of the Kalman rank condition, we refer to \cite[Chapter 1]{Trélat}.}.
\end{lemma}
\begin{proof}
    It is equivalent to prove that $$\text{det}(B_0, A_0 B_0, \cdots, A_0^{N} B)\neq 0.$$
    Note that
    $$
   \text{det}(B_0, A_0 B_0, \cdots, A_0^{N} B)=\text{det}\left(\begin{array}{cccc}
1  &0& \cdots &0 \\
b_1 &a_1+(q-\lambda_1)b_1& \cdots & (q-\lambda_1)^{N-1}\{a_1+(q-\lambda_1)b_1\} \\
\vdots  & \vdots & & \vdots\\
b_N&a_N+(q-\lambda_N)b_N&\cdots&(q-\lambda_N)^{N-1}\{a_N+(q-\lambda_1)b_N\}
\end{array}\right),$$
which leads to
\begin{equation*}
\begin{aligned}
\text{det}(B_0, A_0 B_0, \cdots, A_0^{N} B)&=\prod_{j=1}^{N}\left(a_j+(q-\lambda_j )b_j\right)\text{VdM}(q-\lambda_1,\cdots,q-\lambda_N),
\end{aligned}
\end{equation*}
where $\text{VdM}(q-\lambda_1,\cdots,q-\lambda_N)$ denotes the Van der Monde determinant. 
For any \( 1 \leq j \leq N \), we have
\[
a_j + (q - \lambda_j) b_j = \lambda_j \langle \mathcal{D}(1), \varphi_j \rangle - \langle A \mathcal{D}(1), \varphi_j\rangle.
\]
Integrating by parts, it leads to
\[
a_j + (q- \lambda_j) b_j = -a(0)\partial_x\varphi_j(0).
\]
Hence,
$$
\text{det}(B_0, A_0 B_0, \cdots, A_0^{N} B)=(-a(0))^N(\prod_{j=1}^{N}\partial_x\varphi_j(0)) \text{VdM}(q-\lambda_1,\cdots,q-\lambda_N) .
$$
On one hand, $\text{VdM}(q-\lambda_1,\cdots,q-\lambda_N)\neq 0$, as the eigenvalues $\lambda_j$ are distinct. On the other hand, using the unique continuation property (\cite[Corollary 15.2.2]{TW}), it follows that $\partial_x\varphi_j(0)\neq 0$ for every $j\ge 1$. This completes the proof.
\end{proof}

\noindent \textbf{Selection of $N$}: Given any
$\delta>0$. Since $\lambda_n=\mathcal{O}(n^2)$, there exists $N\ge 2$ such that
\begin{equation}
    \label{N-select}
    q-\lambda_n<-2\delta\quad \text{for every }n>N.
\end{equation}

\noindent \textbf{Controller gain}: Since the pair $ (A_0,B_0)$ satisfies the Kalman rank condition, we can achieve pole-shifting (\cite[Theorem 17]{TW}). Therefore, there exists a matrix $K=[k_0,k_1,\cdots,k_N]\in \mathbb{R}^{1\times (N+1)}$ such that $A_0+B_0K$ is a Hurwitz matrix. In addition, the eigenvalues of $A_0+B_0K$ are less than $-2\delta$.

\subsection{Well-posedness}
    Now, we denote $Y(t)=(u(t),y_1(t),\cdots,y_N(t))\in\mathbb{R}^{1+N}$ and consider the PDE-ODE coupled system    
\begin{equation}
    \label{ode-pde}
    \left\{\begin{array}{ll}
 \dot{w}(t)+A w (t)=q w(t)+q\mathcal{D} u(t)-A \mathcal{D} u(t)-\mathcal{D} (KY(t))+f\left(\mathcal{D}u(t)+w(t)\right)&\text{ for } t\in \mathbb{R}^+,\\[2mm]
\dot{Y}(t)=(A_0+B_0K)Y(t)+R_N(t)&\text{ for } t\in \mathbb{R}^+,\\[2mm]
w(0 )=y_0\quad \text{and} \quad Y(0)=(0,\langle y_0,\varphi_1\rangle,\cdots,\langle y_0,\varphi_N\rangle).
\end{array}\right.
\end{equation}
It is a system with the state \((w, u, y_1,\cdots,y_N)\).
\begin{lemma}
    For each $y_0\in H^1_0(0,1)$, there exists $T=T\left(\|y_0\|_{H^1_0(0,1)}\right)$ such that the system \eqref{ode-pde} admits a unique classical solution $$
    w\in C([0,T);H^1_0(0,1))\cap C^1((0,T);L^2(0,1)),$$ $$Y\in C([0,T);\mathbb{R}^{1\times (N+1)})\cap C^1((0,T);\mathbb{R}^{N+1})$$ satisfying $$w(t)\in H^2(0,1)\cap H^1_0(0,1),\;\; \forall t\in(0,T).$$ Furthermore, if $T<+\infty$, then $\lim\limits_{t\to T^-}(|u(t)|+\|w(t)\|_{H^1_0(0,1)})=+\infty$.
\end{lemma}
\begin{proof}
   We consider the Hilbert spaces $\mathcal{H}$ and $\mathcal{H}^1$, where $\mathcal{H}=L^2(0,1) \times \mathbb{R}^{N+1}$ is endowed with the norm $\|\cdot\|_{\mathcal{H}}^2=\|\cdot\|_{L^2(0,1)}^2+|\cdot|_2^2$, and $\mathcal{H}^1=H^1_0(0,1) \times$ $\mathbb{R}^{N+1}$ is endowed with the norm $\|\cdot\|_{\mathcal{H}^1}^2=\|\cdot\|_{H^1_0(0,1)}^2+|\cdot|_2^2$. Define $\xi(\cdot)=\operatorname{col}\left\{w(\cdot), Y(\cdot)\right\}$. The closed-loop system \eqref{ode-pde} can be presented as
\begin{equation}
    \label{wellpose}
\dot{\xi}(t)+\mathfrak{A}  \xi(t)=F_1(\xi(t))+F_2(\xi(t)),\;\; \forall t>0,
\end{equation}
where
$$
\begin{aligned}
& \mathfrak{A} =\left[\begin{array}{cc}
A & 0 \\
0 & -(A_0+B_0K)
\end{array}\right], \\
& F_1(\xi(t))=\left[\begin{array}{c}
q w(t)+q\mathcal{D} u(t)-A \mathcal{D} u(t)-\mathcal{D} (KY(t)) \\
0
\end{array}\right], \\
& F_2(\xi(t))=\left[f\left(\mathcal{D}u(t)+w(t)\right), R_N(t)\right]^{\top}.
\end{aligned}
$$
The domain of $\mathfrak{A} $ is $\mathcal{D}(\mathfrak{A})=H^2(0,1)\cap H^1_0(0,1) \times \mathbb{R}^{N+1}$. Since $-A$ generates an analytic $C_0$-semigroup on $L^2(0,1)$ and $A_0+B_0K$ is a Hurwitz matrix, $-\mathfrak{A} $ also generates an analytic $C_0$-semigroup on $\mathcal{H}$. 

Fix $\xi_0 \in \mathcal{H}^1$. For $j=1,2$, we let $\xi_j \in B_{{\mathcal{H}^1}}\left(\xi_0, r\right)=\left\{\xi\in \mathcal{H}^1: \|\xi-{\xi}_0\|_{\mathcal{H}^1}\le r\right\}$. Write $\xi_1=\operatorname{col}\left\{w_1, Y_1\right\} $ and $ \xi_2=\operatorname{col}\left\{w_2, Y_2\right\}$. 

\begin{enumerate}
    \item[(i)] For the nonlinear term $F_1(\xi)$, it can be easily verified that
\begin{equation*}
\left\|F_1\left(\xi_1\right)-F_1\left(\xi_2\right)\right\|_{\mathcal{H}} \leq\left(|q|+|q|\|\mathcal{D}\|_{\mathcal{L}}+\|A\mathcal{D}\|_{\mathcal{L}}+\|\mathcal{D}\|_{\mathcal{L}}|K|_2\right)\left\|\xi_1-\xi_2\right\|_{\mathcal{H}^1}.
\end{equation*}

\item[(ii)] For the nonlinear term $F_2(\xi)$, we have
$$
\begin{aligned}
 \left\|F_2\left(\xi_1\right)-F_2\left(\xi_2\right)\right\|_{\mathcal{H}} & \leq 2\left\|f(\mathcal{D}u_1(t)+w_1(t))-f(\mathcal{D}u_2(t)+w_2(t))\right\|_{L^2(0,1)}\\
&\leq  \tilde{C}_1(\|\xi_0\|_{\mathcal{H}^1},r)\|\mathcal{D}u_1(t)+w_1(t)-\mathcal{D}u_2(t)-w_2(t)\|_{H^1_0(0,1)}\\
&\leq  \tilde{C}_2(\|\xi_0\|_{\mathcal{H}^1},r)\left\|\xi_1-\xi_2\right\|_{\mathcal{H}^1}.
\end{aligned}
$$
\end{enumerate}
Here $\tilde{C}_1\left(\|\xi_0\|_{\mathcal{H}^1},r\right)$ and $\tilde{C}_2\left(\|\xi_0\|_{\mathcal{H}^1},r\right)$ are positive constants depending on $\|\xi_0\|_{\mathcal{H}^1}$ and $r$. Since $H^1_0(0,1)= D(A^{\frac{1}{2}})$, the nonlinear term $F_1+F_2$ satisfies {Assumption (F)} in \cite[Theorem 6.3.1 and Theorem 6.3.3]{Pazy}. Therefore, the system \eqref{wellpose} admits a unique classical solution $\xi$ for some $T=T\left(\|y_0\|_{H^1_0(0,1)}\right)$, such that $\xi\in C([0,T);\mathcal{H}^1)\cap C^1((0,T);\mathcal{H})$ and $\xi(t)\in D(\mathfrak{A})$ for every $ t\in(0,T)$. Moreover, if $T<+\infty$, then $\lim\limits_{t\to T^-}(\|\xi(t)\|_{\mathcal{H}^1})=+\infty$. Thus, the proof is completed.
\end{proof}

 We have the following result.
\begin{theorem}[Well-posedness]
\label{Well-posedness-th}
   For each $y_0\in H^1_0(0,1)$, there exists $T=T\left(\|y_0\|_{H^1_0(0,1)}\right)$ such that the system \eqref{controlled-eqs} admits a unique solution $$w\in C([0,T);H^1_0(0,1))\cap C^1((0,T);L^2(0,1)),$$ $$u\in C([0,T))\cap C^1((0,T))$$ satisfying $$w(t)\in H^2(0,1)\cap H^1_0(0,1),\;\; \forall t\in(0,T).$$ Furthermore, if $T<+\infty$, then $\lim\limits_{t\to T^-}(|u(t)|+\|w(t)\|_{H^1_0(0,1)})=+\infty$.
\end{theorem}
\begin{proof}
   We set $w_n(t)=\langle w(t),\varphi_n\rangle$, where $w(t)$ satisfies \eqref{ode-pde}. Firstly, it can be readily derived from the first equation in \eqref{ode-pde} that for every $ n\ge 1$, we have
    \begin{equation*}
        \dot{w}_n(t)+\lambda_nw_n(t)=qw_n(t)+a_nu(t)+b_nKY(t)+f_n(t) ,\;\; \forall t\in (0,T).
    \end{equation*}
    Additionally from the second equation in \eqref{ode-pde}, it follows that for every $ 1\le n\le N$, 
    \begin{equation*}
        \dot{y}_n(t)+\lambda_ny_n(t)=qy_n(t)+a_nu(t)+b_nKY(t)+f_n(t) ,\; \forall t\in (0,T).
    \end{equation*}
   Now, for any $ 1\le n\le N$, we denote $e_n(t)=w_n(t)-y_n(t)$,$\; t\in (0,T)$. It is not difficult to verify that
    \begin{equation*}
       \dot{e}_n(t)=(q-\lambda_n) e_n(t) ,\;\forall  t\in (0,T).
    \end{equation*}
This immediately indicates that $e_n(t)\equiv 0$, since the initial value $e_n(0)=0$.  It is hence deduced that for any $ 1\le n\le N$, we have $$ y_n(t) = \langle w(t),\varphi_n\rangle,\quad \forall t\in[0,T). $$ By replacing $y_n(t)$ with \( \langle w(t),\varphi_n\rangle \) in \eqref{ode-pde}, we conclude that the system \eqref{ode-pde} is equivalent to \eqref{controlled-eqs}, and thus the well-posedness of the system \eqref{controlled-eqs} follows from that of \eqref{ode-pde}.
\end{proof}
\subsection{Rapid stabilization}
Let $\tilde{N}\in \mathbb{N}$ satisfy $\tilde{N}> N$ and denote $Y_{\tilde{N}}(t)=(u(t),w_1(t) ,\cdots ,w_N(t),w_{{N+1}}(t),\cdots ,w_{\tilde{N}}(t))\in \mathbb{R}^{\tilde{N}+1}$. According to Lemma \ref{lemma-mode}, we have
\begin{equation*}
    \label{TITLEN}
    \dot{Y}_{\tilde{N}}(t)=\left(\begin{array}{cc}
A_0+B_0 K & 0 \\
A_1+B_1K & A_2
\end{array}\right)Y_{\tilde{N}}(t)+R_{\tilde{N}}(t),\;\; \forall t\in (0,T),
\end{equation*}
where $$A_1=\left(\begin{array}{cccc}
a_{N+1} & 0 & \cdots&0 \\
a_{N+2} & 0&\cdots& 0 \\
\vdots & \vdots  && \vdots \\
a_{\tilde{N}} & 0 & \cdots&0 
\end{array}\right)\in \mathbb{R}^{(\tilde{N}-N)\times(N+1)},B_1=\left(\begin{array}{c}
b_{N+1} \\
\vdots \\
b_{\tilde{N}}
\end{array}\right),R_{\tilde{N}}(t)=\left(\begin{array}{c}
0\\
f_1(t) \\
\vdots \\
f_{\tilde{N}}(t)
\end{array}\right),$$
 and $$A_2=\left(\begin{array}{cccc}
q-\lambda_{N+1} & 0 & \cdots&0 \\
0 & q-\lambda_{N+2}&\cdots& 0 \\
\vdots & \vdots  &&\vdots \\
0 & 0 & \cdots&q-\lambda_{\tilde{N}}
\end{array}\right)\in \mathbb{R}^{(\tilde{N}-N)\times(\tilde{N}-N)}.$$

Based on the prior selection of $N$ and $K$, both $A_0+B_0K+\delta I$ and $A_2+\delta I$ are Hurwitz matrices. Define $F:=\left(\begin{array}{cc}
A_0+B_0 K & 0 \\
A_1+B_1K & A_2
\end{array}\right)$. Then $F+\delta I$ is a Hurwitz matrix as well. Consequently, there exists a matrix $P\succ 0$ satisfying the Lyapunov equation
\begin{equation}
\label{ly-eq}
    F^TP+PF+2\delta P=-I.
\end{equation}
\begin{lemma}
    \label{eig-of-P}
  There are two constants $\sigma_1=\sigma_1\left(\delta\right)$ and $\sigma_2=\sigma_2\left(\tilde{N},\delta\right)$ so that
    \begin{equation}
        \label{spectrum-P}
        \|P\|_2=\sigma_{\max}(P)\le \sigma_1\quad \text{and}\quad  \sigma_{\min}(P)\ge \sigma_2.
    \end{equation}
\end{lemma}
\begin{proof}
  (i) The proof here is motivated by \cite{KF}. Define two auxiliary matrices
$$F_1=\left(\begin{array}{cc}
A_0+B_0 K+\delta I & 0 \\
0& A_2+\delta I
\end{array}\right),\;\;\;F_2=\left(\begin{array}{cc}
0 & 0 \\
A_1+B_1 K&0
\end{array}\right).$$ 
It follows that $F+\delta I=F_1+F_2$.
Firstly, there is a constant $M_1=M_1\left(\delta\right)\ge 1$ so that
\begin{equation}
\label{F1}
   \|e ^{F_1t}\|_2\le \max\left\{M_1e^{-\delta t},e^{-\delta t}\right\}\le M_1e^{-\delta t},\;\; \forall t>0.
   \end{equation}
Besides, since $|K|_2$ depends only on $\delta$, we can find $M_2=M_2\left(\delta\right)\ge 1$ such that
\begin{equation}
\label{F_2}
    \|F_2\|_2=\|A_1+B_1K\|_2\le \|A_1\|_2+\|B_1K\|_2\le \sum_{n\ge 1}a_n^2+|K|_2+ |K|_2\sum_{n\ge 1}b_n^2<M_2.
\end{equation}
 It is easy to show that for any $t_i\ge 0$ $(i=1,2)$, we have $\prod_{i=1}^2(e^{F_1t_i}F_2)=0$. For each $t>0$, we introduce the identities
\begin{equation}
    \label{vanloan}
    e^{(F+\delta I)t}=e^{F_1t}+\int_{0}^te^{F_1(t-s)}F_2e^{(F+\delta I)s}ds=e^{F_1t}+\int_{0}^te^{F_1(t-s)}F_2e^{F_1s}ds.
\end{equation}
Finally, from \eqref{F1}, \eqref{F_2} and \eqref{vanloan}, we find $M_3=M_3\left(\delta\right)\ge 1$ such that
\begin{equation*}
   \| e^{(F+\delta I)t}\|_2\le M_3e^{-\delta t}(1+t),\;\; \forall t>0.
\end{equation*}
By virtue of tha fact that $P=\int_0^{\infty}e^{(F+\delta I)^{\top}s}e^{(F+\delta I)s}ds$, we deduce that
    \begin{equation*}
    \begin{aligned}
        \sigma_{\max}(P)=\|P\|_2\le \int_0^{+\infty}M_3^2(1+t)^2e^{-2\delta t}dt<\infty.
         \end{aligned}
    \end{equation*}
  We thus define $\sigma_1\left(\delta\right):=\int_0^{+\infty}M_3^2(1+t)^2e^{-2\delta t}dt$, which finalizes the first part of the proof.

 \noindent (ii) Obviously it holds that
   \begin{equation*}
       \left(2(F+\delta I)^{\top}P+\frac{1}{2}I\right)^{\top}\left(2(F+\delta I)^{\top}P+\frac{1}{2}I\right)\succeq 0.
   \end{equation*}
  Expanding this, we derive
   \begin{equation*}
      4P(F+\delta I)(F+\delta I)^{\top}P+P(F+\delta I)+(F+\delta I)^{\top}P+\frac{1}{4}I\succeq 0.
   \end{equation*}
   From the Lyapunov equation \eqref{ly-eq}, it follows that
   \begin{equation*}
       P(F+\delta I)(F+\delta I)^{\top}P\succeq\frac{1}{16}I,
   \end{equation*}
   and thus
   \begin{equation*}
       \sigma_{\min}\left(P(F+\delta I)(F+\delta I)^{\top}P\right)\ge \frac{1}{16}.
   \end{equation*}
   Note that $P(F+\delta I)(F+\delta I)^{\top}P$ is invertible, we have
   \begin{equation*}
       \sigma_{\max}\left((F+\delta I)(F+\delta I)^{\top}\right)\cdot \sigma_{\min}(P^2)\ge \sigma_{\min}\left((F+\delta I)(F+\delta I)^{\top}P^2\right)=\sigma_{\min}\left(P(F+\delta I)(F+\delta I)^{\top}P\right)\ge \frac{1}{16}.
   \end{equation*}
   It is readily verified that
   \begin{equation*}
       \begin{aligned}
           \sigma_{\min}(P)&\ge \frac{1}{4\sqrt{\sigma_{\max}\left((F+\delta I)(F+\delta I)^{\top}\right)}}=\frac{1}{4\|F+\delta I\|_2}\\
           &\ge \frac{1}{4\|F_1\|_2+\|F_2\|_2}\\
           &\ge \frac{1}{4\left(\max\left\{\|A_0+B_0K+\delta I\|_2,\lambda_{\tilde{N}}-q+\delta\right\}+M_2\right)}\\
           &:=\sigma_2\left(\tilde{N},\delta\right).
       \end{aligned}
   \end{equation*}
 This completes  the proof. 
\end{proof}

Define the Lyapunov function
\begin{equation*}
    \label{Lyapunov}
    \mathcal{V}(Y, w)= Y^{\top} P Y+\sum_{n=\tilde{N}+1}^{\infty} \lambda_n \langle w,\varphi_n\rangle^2, \quad \quad \forall (Y, w)\in \mathbb{R}^{\tilde{N}+1}\times H^2(0,1)\cap H^1_0(0,1),
\end{equation*}
where $\tilde{N}>N\ge 2$ will be chosen later. It is not hard to check that $V$ is well defined. {For every $u\in \mathbb{R}$ and $w\in H^2(0,1)\cap H^1_0(0,1)$, we introduce $Y_{\tilde{N}}=(u, \langle w,\varphi_1\rangle,\cdots,\langle w,\varphi_{\tilde{N}}\rangle$) and let $\tilde{\mathcal{V}}(u,w)=\mathcal{V}(Y_{\tilde{N}},w)$.}
\begin{lemma}
\label{h1norm-v}
  The above function $\tilde{\mathcal{V}}$ satisfies that
    \begin{equation}
        \label{VwH1}
       \tilde{\mathcal{V}}(u,w)\sim u^2+\|w\|^2_{H^1_0(0,1)},\quad \forall (u,w)\in \mathbb{R}\times H^2(0,1)\cap H^1_0(0,1).
    \end{equation}
\end{lemma}

\begin{proof}
   It can be straightforwardly verified that
    \begin{equation*}
        \sigma_{\min}(P)u^2+ \min\{\frac{\sigma_{\min}(P)}{\lambda_{\tilde{N}}},1\}\sum_{n=1}^\infty\lambda_n\langle w,\varphi_n\rangle^2\leq \tilde{\mathcal{V}}(u,w)\leq\sigma_{\max}(P)u^2+\max\{\frac{\sigma_{\max}(P)}{\lambda_1},1\}\sum_{n=1}^\infty\lambda_n\langle w,\varphi_n\rangle^2.
    \end{equation*}
We obtain \eqref{VwH1} immediately, since $\sum_{n=1}^\infty\lambda_n\langle w,\varphi_n\rangle^2=\langle A w, w\rangle=\|w\|^2_{D(A^{\frac{1}{2}})}\sim \|w\|^2_{H^1_0(0,1)}$.
\end{proof}
\begin{lemma}
        Let $V(t)=\tilde{\mathcal{V}}(u(t),w(t)),\;t\geq0$. Then  along the trajectory, it holds that
    \begin{equation}
        \label{V-along-traj}
        \begin{aligned}
    \dot{V}(t)+2\delta V(t)&=- Y_{\tilde{N}}^{\top}(t)Y_{\tilde{N}}(t)+2 Y_{\tilde{N}}^{\top}(t)P R_{\tilde{N}}(t)+2 \sum_{n\ge \tilde{N}+1}\lambda_{n} w_{n}(t) \{(q+\delta-\lambda_n )w_n(t)\\
&\quad\quad\quad\quad\quad\quad\quad\quad\quad\quad+a_nu(t)+b_n KY_{N}(t)+f_n(t)\},\;\;\forall t\in(0,T).
    \end{aligned}
    \end{equation}
\end{lemma}
\begin{proof}
    Since $ Y_{\tilde{N}} \in C^{1}((0, T) ; \mathbb{R}^{\tilde{N}+1}) $ and  $w \in C^{1}((0, T) ; L^{2}(0,1))$, it follows that
\begin{equation*}
\begin{aligned}
\dot{V}(t)&= \dot{Y}_{\tilde{N}}^{\top}(t) P Y_{\tilde{N}}(t)+ Y_{\tilde{N}}^{\top}(t) P \dot{Y}_{\tilde{N}}(t)+2\langle A w(t), \dot{w}(t)\rangle-2 \sum_{n=1}^{\tilde{N}} \lambda_{n} w_{n}(t) \dot{w}_{n}(t)\\[2mm]
&= \dot{Y}_{\tilde{N}}^{\top}(t) P Y_{\tilde{N}}(t)+ Y_{\tilde{N}}^{\top}(t) P \dot{Y}_{\tilde{N}}(t)+2 \sum_{n\ge 1}\lambda_{n} w_{n}(t) \dot{w}_{n}(t)-2 \sum_{n=1}^{\tilde{N}} \lambda_{n} w_{n}(t) \dot{w}_{n}(t)\\[2mm]
&= Y_{\tilde{N}}^{\top}(t)(F^T P+PF)Y_{\tilde{N}}(t)+2 Y_{\tilde{N}}^{\top}(t)P R_{\tilde{N}}(t)+2 \sum_{n\ge \tilde{N}+1}\lambda_{n} w_{n}(t) \dot{w}_{n}(t)\\[2mm]
&= Y_{\tilde{N}}^{\top}(t)(F^T P+PF)Y_{\tilde{N}}(t)+2 Y_{\tilde{N}}^{\top}(t)P R_{\tilde{N}}(t)+2 \sum_{n\ge \tilde{N}+1}\lambda_{n} w_{n}(t) \{(q-\lambda_n )w_n(t)\\
&\quad\quad\quad\quad\quad\quad\quad\quad\quad\quad\quad\quad\quad\quad+a_nu(t)+b_n KY_{N}(t)+f_n(t)\},\;\;\forall t\in(0,T).
\end{aligned}
\end{equation*}
Using the property of the Lyapunov equation \eqref{ly-eq}, \eqref{V-along-traj} follows.
\end{proof}

\begin{lemma}\label{Tilde-N-Select}
    Let $S_{1,\tilde{N}}=\sum\limits_{n\ge \tilde{N}+1}a_n^2$ and $S_{2,\tilde{N}}=\sum\limits_{n\ge \tilde{N}+1}b_n^2|K|_2$. Then there exists $\tilde{N}>N$ such that 
\begin{equation}
\label{TILDE-N-Select}
\left\{\begin{aligned}
    &(-1+ 2S_{1,\tilde{N}}+2S_{2,\tilde{N}}+\frac{1}{\tilde{N}}) I +\frac{P^TP}{\tilde{N}}\preceq 0,\\[2mm]
   & 2q+2\delta-\lambda_n+\frac{1}{\tilde{N}}\lambda_n+\frac{1}{\tilde{N}}\le 0,\;\;\forall n\ge \tilde{N}+1.
    \end{aligned}\right.
\end{equation}
\end{lemma}
\begin{proof}
    \noindent On one hand,
in view of the fact that $\|P\|_2= \mathcal{O}(1)$ uniformly in $\tilde{N}$, there exists $N_1> N$ such that 
\begin{equation*}
    \sigma_{\max} ((-1+ {2}S_{1,\tilde{N}}+{2}S_{2,\tilde{N}}+\frac{1}{\tilde{N}}) I +\frac{1}{\tilde{N}}P^TP)\le 0 \quad \text{if $\tilde{N}\ge N_1 $}.
\end{equation*}
On the one hand, since $\lambda_n=\mathcal{O}(n^2)$, there is $N_2> N$ so that
\begin{equation*}
    2q+2\delta+1-\frac{1}{2}\lambda_n\le 0\quad \text{for every }n\ge N_2.
\end{equation*}
We let $\hat{N}=\max\{N_1,N_2\}$, it is not hard to check that for every $\tilde{N}>\hat{N}$, \eqref{TILDE-N-Select} holds true.
\end{proof}
\begin{theorem}
\label{th2}
   Assume that $f$ satisfies the assumption (H). Given any $\delta>0$, there exist constants $\rho=\rho\left(\delta\right)$ and $\hat{M}=\hat{M}\left(\delta\right)\ge 1$ such that for any initial data $y_0\in H^1_0(0,1)$ satisfying $$\| y_0\|_{ H^1_0(0,1)}\le \rho,$$
the solution $(u,w)$ of the system \eqref{controlled-eqs} verifies 
\begin{equation}
    \label{ex-sta33}
    \begin{aligned}
               |u(t)|+ \|w(t)\|_{H^1_0(0,1)}\le \hat{M}e^{-\delta t}\|y_0\|_{H^1_0(0,1)},\;\;\forall t>0.
    \end{aligned}
\end{equation}
\end{theorem}
\begin{proof}
The proof of {Theorem \ref{th2}} is divided into several steps. It should be noted that the system \eqref{controlled-eqs} is well-posed over $[0,T)$ for some $T=T\left(\| y_0\|_{ H^1_0(0,1)}\right)$ according to {Theorem \ref{Well-posedness-th}}.

\noindent \textbf{Step 1.} We derive a set of linear matrix inequalities (LMIs) that ensure the exponential stability of $V(t)$ over $[0,T)$.

By using Young’s inequality, for every positive number $\alpha_1,\alpha_2$ and $\alpha_3$, we have
\begin{equation*}
    2 \sum_{n\ge \tilde{N}+1}\lambda_{n}a_n w_{n}(t)u(t)\le  \sum_{n\ge \tilde{N}+1}\{\alpha_1a_n^2u(t)^2+\frac{1}{\alpha_1}\lambda_{n}^2w_{n}^2(t)\},\;\;\forall t\in (0,T),
\end{equation*}
\begin{equation*}
    2 \sum_{n\ge \tilde{N}+1}\lambda_{n}b_n w_{n}(t)KY_{N}(t)\le  \sum_{n\ge \tilde{N}+1}\{\alpha_2b_n^2|K|_2^2Y_{N}^{\top}(t)Y_{N}(t)+\frac{1}{\alpha_2}\lambda_{n}^2w_{n}^2(t)\},\;\;\forall t\in (0,T),
\end{equation*}
and
\begin{equation*}
\label{f-est}
\begin{aligned}
    2 \sum_{n\ge \tilde{N}+1}\lambda_{n}w_{n}(t)f_n(t)&\le  \sum_{n\ge \tilde{N}+1}\{\alpha_3f_n^2(t)+\frac{1}{\alpha_3}\lambda_{n}^2w_{n}^2(t)\}\\[2mm]
    &=\frac{1}{\alpha_3} \sum_{n\ge \tilde{N}+1}\lambda_{n}^2w_{n}^2(t)+\alpha_3\|f(\mathcal{D}u(t)+w(t))\|^2_{L^2(0,1)}-\alpha_3R_{\tilde{N}}^{\top}(t)R_{\tilde{N}}(t),\;\;\forall t\in (0,T).
    \end{aligned}
\end{equation*}

Then employing the assumption (H) on $f$, we derive that
\begin{equation*}
\begin{aligned}
\|f(\mathcal{D}u(t)+w(t))\|^2_{L^2(0,1)}&\le c^2(\|\mathcal{D}u(t)+w(t)\|_{L^{\infty}(0,1)})\|\mathcal{D}u(t)+w(t)\|^2_{H^{1}(0,1)},\;\;\forall t\in (0,T).
 \end{aligned}
 \end{equation*}
For $C_0=2\max\left\{k_2^2,\|\mathcal{D}\|^2_{\mathcal{L}}\right\}$, we have
\begin{equation*}
\begin{aligned}
   2 \sum_{n\ge \tilde{N}+1}\lambda_{n}w_{n}(t)f_n(t)&\le\frac{1}{\alpha_3} \sum_{n\ge \tilde{N}+1}\lambda_{n}^2w_{n}^2(t)+\alpha_3 c^2(\|\mathcal{D}u(t)+w(t)\|_{L^{\infty}(0,1)})\|\mathcal{D}u(t)+w(t)\|^2_{H^{1}(0,1)}\\
&\quad\quad\quad\quad\quad\quad\quad\quad\quad\quad\quad\quad\quad\quad\quad\quad\quad\quad\quad\quad\quad\quad\quad-\alpha_3R_{\tilde{N}}^{\top}(t)R_{\tilde{N}}(t)\\
   &\le \frac{1}{\alpha_3} \sum_{n\ge \tilde{N}+1}\lambda_{n}^2w_{n}^2(t)+\alpha_3C_0 c^2(\|\mathcal{D}u(t)+w(t)\|_{L^{\infty}(0,1)})\{u^2(t)+\sum_{n\ge 1}\lambda_n w_n^2(t)\}\\
&\quad\quad\quad\quad\quad\quad\quad\quad\quad\quad\quad\quad\quad\quad\quad\quad\quad\quad\quad\quad\quad\quad\quad-\alpha_3R_{\tilde{N}}^{\top}(t)R_{\tilde{N}}(t)\\
   &\le \frac{1}{\alpha_3}\sum\limits_{n\ge \tilde{N}+1}\lambda_{n}^2w_{n}^2(t) +\alpha_3C_0 c^2(\|\mathcal{D}u(t)+w(t)\|_{L^{\infty}(0,1)})\{u^2(t)+\sum_{n\ge \tilde{N}+1}\lambda_n w_n^2(t)\\
   &\quad\quad\quad\quad\quad\quad\quad\quad\quad+\lambda_{\tilde{N}}Y_{\tilde{N}}^{\top}(t)Y_{\tilde{N}}(t)\}-\alpha_3R_{\tilde{N}}^{\top}(t)R_{\tilde{N}}(t),\;\;\forall t\in (0,T).
    \end{aligned}
\end{equation*}

Hence, we obtain that
\begin{equation*}
    \dot{V}(t)+2\delta V(t)\le Y_{\tilde{N}}^{\top}(t)\Phi Y_{\tilde{N}}(t)+2 Y_{\tilde{N}}^{\top}(t)P R_{\tilde{N}}(t)-\alpha_3R_{\tilde{N}}^{\top}(t)R_{\tilde{N}}(t)+\sum_{n\ge \tilde{N}+1}\lambda_{n} \Psi_n w_{n}^2(t),\;\;\forall t\in (0,T),
\end{equation*}
where 
\begin{equation*}
    \begin{aligned}
    &\Phi=\{-1+ \alpha_1S_{1,\tilde{N}}+\alpha_2S_{2,\tilde{N}}+\alpha_3C_0 c^2(\|\mathcal{D}u(t)+w(t)\|_{L^{\infty}(0,1)})(1+\lambda_{\tilde{N}})\} I ,\\[2mm]
&\Psi_n=2q+2\delta-2\lambda_n+\frac{1}{\alpha_1}\lambda_n+\frac{1}{\alpha_2}\lambda_n+\frac{1}{\alpha_3}\lambda_n+\alpha_3C_0 c^2(\|\mathcal{D}u(t)+w(t)\|_{L^{\infty}(0,1)}),\;\;\forall n\ge \tilde{N}+1.
    \end{aligned}
\end{equation*}
It is equivalent to
\begin{equation*}
    \label{LMIs}
    \dot{V}(t)+2\delta V(t)\le \begin{pmatrix}
 Y_{\tilde{N}}(t)\\
R_{\tilde{N}}(t)
\end{pmatrix}^T\Theta\begin{pmatrix}
 Y_{\tilde{N}}(t)\\
R_{\tilde{N}}(t)
\end{pmatrix}+\sum_{n\ge \tilde{N}+1}\lambda_{n} \Psi_n w_{n}^2(t),\;\;\forall t\in (0,T),
\end{equation*}
where 
\begin{equation*}
    \label{Theta}
    \Theta:=\begin{pmatrix}
 \Phi &  P\\
 P  & -\alpha_3
\end{pmatrix}.
\end{equation*}

If the LMIs $\Theta\preceq 0$ and $\Psi_n\le 0 \text{ for every }n\ge \tilde{N}+1$, then $\dot{V}(t)+2\delta V(t)\le 0$ over $(0,T)$. Therefore, we derive a verification condition to ensure the exponential stability of $V(t)$, i.e.,
\begin{equation}
\label{theta and psi<0}
\left\{\begin{aligned}
    &\Theta\preceq 0 \text{:  }\{-1+ \alpha_1S_{1,\tilde{N}}+\alpha_2S_{2,\tilde{N}}+\alpha_3C_0 c^2(\|\mathcal{D}u(t)+w(t)\|_{L^{\infty}(0,1)})(1+\lambda_{\tilde{N}})\} I +\frac{P^TP}{\alpha_3}\preceq 0,\\[2mm]
   & \Psi_n\le 0 \text{:  }  2q+2\delta-2\lambda_n+\frac{1}{\alpha_1}\lambda_n+\frac{1}{\alpha_2}\lambda_n+\frac{1}{\alpha_3}\lambda_n+\alpha_3C_0 c^2(\|\mathcal{D}u(t)+w(t)\|_{L^{\infty}(0,1)})\le 0.
    \end{aligned}\right.
\end{equation}

\noindent \textbf{Step 2.} We prove that if $\|y_0\|_{H^1_0(0,1)}$ is small enough, then the above condition \eqref{theta and psi<0} is satisfied for some $\tilde{N}$ and $\alpha_i$, $i=1,2,3$.

  Let $\alpha_1=\alpha_2=2$ and $\alpha_3=\tilde{N}$. It is equivalent to prove the following holds 
\begin{equation}
\label{equivalent-LMIs}
\left\{\begin{aligned}
    &\{-1+ 2S_{1,\tilde{N}}+2S_{2,\tilde{N}}+\tilde{N}(1+\lambda_{\tilde{N}}) C_0c^2(\|\mathcal{D}u(t)+w(t)\|_{L^{\infty}(0,1)})\} I +\frac{P^TP}{\tilde{N}}\preceq 0,\\[2mm]
   & 2q+2\delta-\lambda_n+\frac{1}{\tilde{N}}\lambda_n+\tilde{N}C_0 c^2(\|\mathcal{D}u(t)+w(t)\|_{L^{\infty}(0,1)})\le 0,\;\;\forall n\ge \tilde{N}+1.
    \end{aligned}\right.
\end{equation}
Firstly, using the properties of $c(\cdot)$, it is not hard to find a constant $\sigma=\sigma(\tilde{N})$ such that
\begin{equation}\label{sigma}
   \tilde{N}(1+\lambda_{\tilde{N}}) C_0 c^2(\sigma)<\frac{1}{\tilde{N}}.
\end{equation}
Secondly, we show that if $y_0\in H^1_0(0,1)$ satisfies 
\begin{equation}
\label{rho}
    \|y_0\|_{H^1_0(0,1)}\le\rho:=\frac{\sigma}{\tilde{c}_2^2\|\mathcal{D}\|_{\mathcal{L}}\hat{M}},
\end{equation}
then it holds that
\begin{equation*}
\label{bound}
   \tilde{c}_2\|w(t)\|_{H^1_0(0,1)}+\tilde{c}_2\|\mathcal{D}\|_{\mathcal{L}}|u(t)| <{\sigma} ,\;\; \forall t\in [0,T).
  \end{equation*}
By contradiction, i.e., we suppose that there exists $\tilde{t}\in [0,T)$ such that $   \tilde{c}_2\|\mathcal{D}\|_{\mathcal{L}}|u(\tilde{t})|+   \tilde{c}_2\|w(\tilde{t})\|_{H^1_0(0,1)}\ge \sigma$. When $t=0$, we have
    \begin{equation*}
           \tilde{c}_2\|w(0)\|_{H^1_0(0,1)}+   \tilde{c}_2\|\mathcal{D}\|_{\mathcal{L}}|u(0)|\le \frac{{\sigma}}{   \tilde{c}_2\|\mathcal{D}\|_{\mathcal{L}}\hat{M}}< {\sigma}.
    \end{equation*}
    Thus, we can define $0<t^*\le \tilde{t}$ the shortest time that satisfies
  \begin{equation*}
        \tilde{c}_2\|w(t)\|_{H^1_0(0,1)}+   \tilde{c}_2\|\mathcal{D}\|_{\mathcal{L}}|u(t)|\ge {\sigma}.
  \end{equation*}
Using the fact that $w\in C([0,T);H^1_0(0,1))$ and $u\in C([0,T))$, we have $   \tilde{c}_2\|w(t^*)\|_{H^1_0(0,1)}+   \tilde{c}_2\|\mathcal{D}\|_{\mathcal{L}}|u(t^*)|={\sigma}$ and $   \tilde{c}_2\|w(t)\|_{H^1_0(0,1)}+   \tilde{c}_2\|\mathcal{D}\|_{\mathcal{L}}|u(t)|<{\sigma}$ for every $t\in [0,t^*)$. This implies that for any $t\in [0,t^*)$, we have
  \begin{equation*}
     \tilde{N}C_0c^2(\|\mathcal{D}u(t)+w(t)\|_{L^{\infty}(0,1)})<\tilde{N}(1+\lambda_{\tilde{N}}) C_0 c^2(\|\mathcal{D}u(t)+w(t)\|_{L^{\infty}(0,1)})< \tilde{N}(1+\lambda_{\tilde{N}}) C_0 c^2(\sigma)<\frac{1}{\tilde{N}}.
  \end{equation*}
  Based on {Lemma \ref{Tilde-N-Select}}, there exists $\tilde{N}$ such that \eqref{TILDE-N-Select} holds. Therefore, \eqref{equivalent-LMIs} is satisfied, and $V(t)$ is exponentially stable over $[0,{t}^*)$.
 Using {Lemma \ref{h1norm-v}}, it leads to
\begin{equation*}
\begin{aligned}
   k_1\min\left\{\sigma_{\min}(P),\min\{\frac{\sigma_{\min}(P)}{\lambda_{\tilde{N}}},1\}\right\} &(u^2(t)+\|w(t)\|^2_{H^1_0(0,1)})\le V(t)\le e^{-2\delta t}V(0)\\
   &\le k_2e^{-2\delta t}\max\left\{\sigma_{\max}(P),\max\{\frac{\sigma_{\max}(P)}{\lambda_1},1\}\right\}\|y_0\|^2_{H^1_0(0,1)},\;\forall t\in [0,{t}^* ).
   \end{aligned}
\end{equation*}
According to {Lemma \ref{eig-of-P}}, the above inequalities readily lead to
\begin{equation*}
\begin{aligned}
    u^2(t)+\|w(t)\|^2_{H^1_0(0,1)}\le \frac{k_2\max\left\{\sigma_1,\max\{\frac{\sigma_1}{\lambda_1},1\}\right\}}{k_1\min\left\{\sigma_2,\min\{\frac{\sigma_2}{\lambda_{\tilde{N}}},1\}\right\}}e^{-2\delta t}\|y_0\|^2_{H^1_0(0,1)}, \;\forall t\in [0,{t}^*).
    \end{aligned}\end{equation*}
    We define $\hat{M}=\left(\frac{k_2\max\left\{\sigma_1,\max\{\frac{\sigma_1}{\lambda_1},1\}\right\}}{k_1\min\left\{\sigma_2,\min\{\frac{\sigma_2}{\lambda_{\tilde{N}}},1\}\right\}}\right)^{\frac{1}{2}}$. It indicates that 
\begin{equation*}
                u(t)+\|w(t)\|_{H^1_0(0,1)}\le\hat{M}e^{-\delta t}\|y_0\|_{H^1_0(0,1)},\;\;\forall t\in [0,{t}^* ).
\end{equation*}
Hence, using the continuity again, we obtain that
  \begin{equation*}
         \tilde{c}_2\|w(t)\|_{H^1_0(0,1)}+   \tilde{c}_2\|\mathcal{D}\|_{\mathcal{L}}|u(t)|\le    \tilde{c}_2\|\mathcal{D}\|_{\mathcal{L}}\hat{M}e^{-\delta t}\|y_0\|_{H^1_0(0,1)}\le \frac{\sigma}{   \tilde{c}_2}<{\sigma} ,\;\; \forall t\in [0,t^*],
  \end{equation*}
  while contradicting the definition of $t^*$.

Finally, under the condition \eqref{rho}, it follows that
\begin{equation*}
         \tilde{c}_2\|w(t)\|_{H^1_0(0,1)}+   \tilde{c}_2\|\mathcal{D}\|_{\mathcal{L}}|u(t)| <{\sigma},\;\; \forall t\in [0,T).
  \end{equation*}
Therefore, for any $t\in [0,T)$, we have
  \begin{equation*}
     \tilde{N}C_0c^2(\|\mathcal{D}u(t)+w(t)\|_{L^{\infty}(0,1)})<\tilde{N}(1+\lambda_{\tilde{N}}) C_0 c^2(\|\mathcal{D}u(t)+w(t)\|_{L^{\infty}(0,1)})< \tilde{N}(1+\lambda_{\tilde{N}}) C_0 c^2(\sigma)<\frac{1}{\tilde{N}}.
  \end{equation*}
This, combined with the selection of $\tilde{N}$ and \eqref{TILDE-N-Select} again, implies \eqref{equivalent-LMIs} at once.

  \noindent \textbf{Step 3.} We prove that the solution $(u,w)$ can be extended to $[0,+\infty)$ when $y_0$ satisfies \eqref{rho}.
  
According to the results established in {Step 2}, it implies that $\|w(t)\|_{H^1_0(0,1)}+|u(t)|$ is uniformly bounded over $t\in [0,T)$. Thus, by virtue of the last fact in {Theorem \ref{Well-posedness-th}}, the solution can be extended to $[0,+\infty)$. 
 
\noindent \textbf{Step 4.} We prove that the constants $\rho$ and $\hat{M}$ depend only on $\delta$.

(i) From the proof of {Lemma \ref{Tilde-N-Select}}, it can be readily verified that $\tilde{N}$ depends only on $\delta$. Therefore, $\lambda_2(\tilde{N},\delta)$ is a constant depending on $\delta$ as well. Given the construction of \(\hat{M}\) presented above, it follows that $\hat{M}$ depends only on $\delta$.

 (ii) From \eqref{sigma}, we observe that $\sigma$ depends on $\tilde{N}$ and thereby also on $\delta$. By letting $\rho=\frac{\sigma}{\tilde{c}_2^2\|\mathcal{D}\|_{\mathcal{L}}\hat{M}}$, we conclude the proof.
\end{proof}

\begin{proofof}{Theorem \ref{main-th1}}
{Define a constant $k=k_0+\sum_{j=1}^Nk_jb_j$ and a feedback operator
\begin{equation}
    \label{feedback-k}
\mathcal{K}y=\sum_{j=1}^Nk_j \langle y,\varphi_j\rangle,\quad \forall y\in L^2(0,1).
\end{equation}
It follows that $\mathcal{K}:L^2(0,1)\to \mathbb{R}$. Let the control $u:\mathbb{R}^+\to \mathbb{R}$ to be such that
$$\dot{u}(t)=v(t)=ku(t)+\mathcal{K}(y(t)),\;\;\forall t>0,$$ i.e.,
\begin{equation*}
    \label{control-feedback-law}
    \left\{\begin{aligned}
        &\dot{u}(t)=k_0u(t)+\sum_{j=1}^Nk_j \left(\langle y(t),\varphi_j\rangle-\langle\mathcal{D}u(t),\varphi_j\rangle\right)\quad\text{for }t\in\mathbb{R}^+,\\
        &u(0)=0,
    \end{aligned}\right.
\end{equation*}
and let $w(t)=y(t)-\mathcal{D}u(t)$ for $ t\ge 0$. It can be readily deduced that the controlled system \eqref{Dirichlet-sys} is equivalent to the system \eqref{controlled-eqs}.
Thus, according to {Theorem \ref{Well-posedness-th}}, there exists $T=T\left(\|y_0\|_{H^1_0(0,1)}\right)$ such that \eqref{Dirichlet-sys} is well-posed on $[0, T)$. Let $y_0\in H^1_0(0,1)$ satisfy $\|y_0\|_{H^1_0(0,1)}\le \rho$, where $\rho$ is defined in \eqref{rho}. From {Theorem \ref{th2}}, it follows that the solution can be extended to $[0,+\infty)$ and satisfies \eqref{ex-sta33}. In view of $y(t)=w(t)+\mathcal{D}u(t)$, we have
$$\|y(t)\|_{H^1(0,1)}\le\|w(t)\|_{H^1_0(0,1)}+\|\mathcal{D}\|_{\mathcal{L}} |u(t)|\le \|\mathcal{D}\|_{\mathcal{L}}\hat{M}e^{-\delta t}\|y_0\|_{H^1_0(0,1)},\;\;\forall t>0.$$
By letting $M=\|\mathcal{D}\|_{\mathcal{L}}\hat{M}$, we complete the proof of {Theorem \ref{main-th1}}.}
\end{proofof}
\begin{remark}
 Let $\theta_1,\theta_2$ satisfy one of the following conditions
        \begin{center}
             (1) $\theta_1=\theta_2=\frac{\pi}{2}$ (Neumann boundary),\quad(2)  $\theta_1=\frac{\pi}{2},$ $ \theta_2=0$ (Mixed boundary I),\\
            (3) $\theta_1=\theta_2=0$ (Dirichlet boundary),  \quad(4) $\theta_1=0,$ $ \theta_2=\frac{\pi}{2}$ (Mixed boundary II).
        \end{center} Consider the nonlinear parabolic equation with a boundary control
\begin{equation}\label{parabolic-sys}
\left\{\begin{array}{ll}
 \partial_t y(t,x) =\partial_x (a(x)\partial_x y( t,x))+c(x)y( t,x) +qy( t,x)+f(y( t,x))&\text{ for }(t,x)\in \mathbb{R}^+\times (0,1),\\[2mm]
\cos(\theta_1)y(t,0)-\sin(\theta_1)\partial_x y(t,0)=u(t)&\text{ for }t\in \mathbb{R}^+,\\[2mm]
\cos(\theta_2)y(t,1)+\sin(\theta_2)\partial_x y(t,1)=0&\text{ for }t\in \mathbb{R}^+,\\[2mm]
y(0,x)=y_0(x )&\text{ for }x\in (0,1).
\end{array}\right.
 \end{equation}
Introduce the Sturm-Liouville operator $A$,
\begin{equation*}\label{S-L}
       Af=-(\partial_x (a(x)\partial_x f)+c(x)f), 
        \end{equation*}
        with the domain $D( A)=\left\{f\in H^2(0,1):\cos(\theta_1)f(0)-\sin(\theta_1)\partial_x f(0)=0,\cos(\theta_2)f(1)+\sin(\theta_2)\partial_x f(1)=0\}\right.$.  
       Note that $D(A^{\frac{1}{2}})\subset H^1(0,1)$ {and $\|\cdot\|_{D(A^{\frac{1}{2}})}\sim \|\cdot\|_{H^1(0,1)}.$}
       
        Firstly, given $u\in \mathbb{R}$, it is not difficult to find $\mathcal{D}u\in H^1(0,1)$ satisfying
        $$\left\{\begin{array}{ll}
&\cos(\theta_1)(\mathcal{D}u)(0)-\sin(\theta_1)\partial_x (\mathcal{D}u)(0)=u,\\[2mm]
&\cos(\theta_2)(\mathcal{D}u)(1)+\sin(\theta_2)\partial_x (\mathcal{D}u)(1)=0,
\end{array}\right.$$and let $w(t)=y(t)-\mathcal{D} u(t),\; t\ge 0$. Similar to the proof of Theorem \ref{main-th1}, we have the following corollary.
        \begin{corollary}
     \label{co1}
  Assume that $f$ satisfies the assumption (H). Given any $\delta>0$, there are constants $\rho>0$, $k\in \mathbb{R}$ and $M\ge 1$, as well as a feedback operator $\mathcal{K}:L^2(0,1)\to\mathbb{R}$, all depending only on $\delta$, such that for every $y_0\in D(A^{\frac{1}{2}})$ satisfying
         \begin{equation*}
             \|y_0\|_{H^1(0,1)}\le \rho,
         \end{equation*}
        the system \eqref{parabolic-sys} with $u(t)=\int_0^t e^{k(t-s)}\mathcal{K}(y(s))ds$ admits a unique solution $$y\in C([0,+\infty);H^1(0,1))\cap 
 C^1(\mathbb{R}^+;L^2(0,1))  $$ verifying the following exponential stability
         \begin{equation*}
             \|y(t)\|_{H^1(0,1)}\le Me^{-\delta t}\|y_0\|_{H^1(0,1)},\;\;\forall t>0.
         \end{equation*}
\end{corollary}
\end{remark}

\section{Application}\label{sec5}
 This section is divided into two parts, which are focused on applying {Theorem \ref{main-th1}} to the Burgers equation and the Allen-Cahn equation, respectively. We first show that the nonlinear term of both equations satisfies the assumption (H), which ensures locally rapid stabilization. Next, using certain energy estimate inequalities, we deduce that the $H^1$-norm of the uncontrolled system decays. Without control, the system tends to converge to a region where the rapid stabilization is guaranteed. Thereafter, by applying Theorem \ref{main-th1}, we can achieve globally rapid stabilization.

 
\subsection{The Burgers equation}

Consider the Burgers equation 
\begin{equation}
    \label{burgers-eq}
\left\{\begin{array}{ll}
\partial_t y(t,x)-\partial_{xx}y(t,x)+y (t,x) \partial_x y(t,x)=0 & \text{for }(t,x)\in (0,T)\times (0,1),\\[2mm]
y( t,0)=y( t,1)=0& \text{for }t\in (0,T),\\[2mm]
y(t_0,x)=y_0(x)& \text{for }x\in (0,1).
\end{array}\right.
\end{equation}

\begin{lemma}
    \label{thm1.1-to-burs}
    The nonlinear term $f(y)=-y\partial_x y$ in \eqref{burgers-eq} satisfies the assumption (H). 
\end{lemma}
\begin{proof}
It is not hard to check that for every $y_1,y_2\in B_{{H^1}}(y_0,r)$, the following holds
    \begin{equation*}
             \label{bf}
             \begin{aligned}
                 \|f(y_1)-f(y_2)\|_{L^2(0,1)}   &\le \|y_1\partial_x y_1-y_1\partial_x y_2\|_{L^2(0,1)}+\|y_1\partial_x y_2-y_2\partial_x y_2\|_{L^2(0,1)}\\
                 &\le \|y_1\|_{L^{\infty}(0,1)}\|\partial_x y_1-\partial_x y_2\|_{L^2(0,1)}+\|y_1-y_2\|_{L^{\infty}(0,1)}\|\partial_x y_2\|_{L^2(0,1)}\\
                 &\le 2\tilde{c}_2(r+\|y_0\|_{H^1(0,1)})\|y_1-y_2\|_{H^1(0,1)}.
             \end{aligned}
         \end{equation*}
       Moreover, we have
       \begin{equation*}
             \label{bf1}
                \|f(y)\|_{L^2(0,1)}= \|y\cdot \partial_x y\|_{L^2(0,1)}\le \|y\|_{L^{\infty}(0,1)}\|y\|_{H^1(0,1)},\; \forall y\in H^1(0,1).
         \end{equation*}
        Therefore, the assumption (H) is satisfied with $C(\|y_0\|_{H^1(0,1)},r)=2\tilde{c}_2(r+\|y_0\|_{H^1(0,1)})$ and $c(x)=x$.
\end{proof}

The following well-known result holds (see \cite{S}).
\begin{lemma}
    (i) If $y_0\in L^2(0,1)$, then the equation \eqref{burgers-eq} admits a unique mild solution 
        \begin{equation*}
            \label{Burgers-wellposed1}
            y\in L^{\infty}(0,T;L^2(0,1))\cap L^2(0,T;H^1_0(0,1))\cap H^1(0,T;H^{-1}(0,1)).
        \end{equation*}
        Moreover, $y(t)\in H^2(0,1)\cap H^1_0(0,1)$ for every $t\in (0,T)$.
        
         (ii) If $y_0\in H^1_0(0,1)$, then 
         \begin{equation*}
            \label{Burgers-wellposed3}
            y\in L^{2}(0,T;H^2(0,1))\cap L^{\infty}(0,T;H^1_0(0,1))\cap H^1(0,T;L^2(0,1)).
        \end{equation*}
        Moreover, the solution coincides with the classical sense.
        
      (iii) Furthermore, if $y_0\in H^2(0,1)\cap H^1_0(0,1)$, then 
         \begin{equation*}
            \label{Burgers-wellposed4}
            y\in L^{\infty}(0,T;H^2(0,1))\cap H^1(0,T;H^1_0(0,1))\cap H^2(0,T;H^{-1}(0,1)).
        \end{equation*}
\end{lemma}

The following energy estimates play an important role in proving the dissipativity of the \(H^1\)-norm. We outline the proof in {Appendix}.

\begin{lemma}\label{energy-estimatet}
For each $\tau,t \in\left(0, T\right)$ and $y_0\in L^2(0,1)$, there are constants $C_1=C_1\left(\left\|y_0\right\|_{L^2(0,1)}, t\right),C_2=C_2\left(\|y(\tau)\|_{H^1(0,1)}\right),C_3=C_3\left(\|y(\tau))\|_{H^1(0,1)},t,\tau\right)$ and $C_4=C_4\left(\|y(\tau))\|_{H^2(0,1)}\right)$ so that
\begin{enumerate}
    \item[(i)]For $t>0$,
    \begin{equation}
        \label{energy estimate1}
        \|y(t)\|^2_{L^2(0,1)} \leq e^{-2 c_1t}\left\|y_0\right\|^2_{L^2(0,1)},
    \end{equation}
    \begin{equation}
        \label{energy estimate3}
      \left\|y(t)\right\|^2_{H^1(0,1)} \le C_1\left(\left\|y_0\right\|_{L^2(0,1)}, t\right).
    \end{equation}

    \item[(ii)]For $t>\tau$,
    \begin{equation}
        \label{energy estimate4}
       \left\|y(t)\right\|^2_{H^1(0,1)} \le C_2\left(\|y(\tau)\|_{H^1(0,1)}\right),
    \end{equation}

    \begin{equation}
        \label{energy estimate8}
        \left\|y(t)\right\|^2_{H^2(0,1)} \le C_3\left(\|y(\tau))\|_{H^1(0,1)},t,\tau\right).
    \end{equation}
   
    \begin{equation}
        \label{energy estimate10}
      \left\|y(t)\right\|^2_{H^2(0,1)} \leq C_4\left(\|y(\tau))\|_{H^2(0,1)}\right).
    \end{equation}
\end{enumerate}
\end{lemma}

\begin{lemma}
\label{diss-property}
    For each $\tau>0$ and $y_0\in L^2(0,1)$, there is a constant $L=L\left(\|y_0\|_{L^2(0,1)},\tau\right)$ so that the solution of the Burgers equation \eqref{burgers-eq} satisfies the following dissipative property
     \begin{equation}
         \label{diss-property-burgers}
         \| y(t)\|_{H^1(0,1)}\le L\left(\|y_0\|_{L^2(0,1)},\tau\right)e^{-\frac{c_1 }{2}t},\;\;\forall t>\tau.
     \end{equation}
\end{lemma}
\begin{proof}
    Firstly, from the Gagliardo-Nirenberg interpolation inequality \eqref{g-n}, we have
    \begin{equation}
        \label{diss-est1}
        \|\partial_x y(t)\|_{L^2(0,1)}\le c_3\|y(t)\|_{L^2(0,1)}^{\frac{1}{2}}\|y(t)\|^{\frac{1}{2}}_{H^2(0,1)},\;\;\forall t>0.
    \end{equation}
    Then, using the energy estimates \eqref{energy estimate1} and \eqref{energy estimate10}, we derive that
    \begin{equation}
        \label{diss-est2}
          \|\partial_x y(t)\|_{L^2(0,1)}\le c_3e^{-\frac{c_1 }{2}t}\|y_0\|_{L^2(0,1)}^{\frac{1}{2}}\tilde{C}\left(\|y(\tau)\|_{H^2(0,1)}\right),\;\;\forall t>\tau
    \end{equation}
   for some constant $\tilde{C}=\tilde{C}\left(\|y(\tau)\|_{H^2(0,1)}\right)$.

   Recall the auxiliary estimates $$\| y(\frac{\tau}{2})\|_{H^1(0,1)}^2\le C_1\left(\|y_0\|_{L^2(0,1)},\frac{\tau}{2}\right),\,\|y(\tau)\|^2_{H^2(0,1)}\le C_3\left(\| y(\frac{\tau}{2})\|_{H^1(0,1)},\tau,\frac{\tau}{2}\right),$$along with \eqref{diss-est2}, it readily leads to \eqref{diss-property-burgers}.
\end{proof}

\begin{theorem}
\label{global-exp-sta-burgers}
    Consider the Burgers equation with a boundary control
    \begin{equation}
    \label{Burgers-eq2}
\left\{\begin{array}{ll}
\partial_t y(t,x)- \partial_{xx}y(t,x)+y (t,x) \partial_x y(t,x)=0 & \text{for }(t,x)\in \mathbb{R}^+\times (0,1),\\[2mm]
y( t,0)=u(t)\quad\text{and}\quad y( t,1)=0& \text{for }t\in \mathbb{R}^+,\\[2mm]
y(t_0,x)=y_0(x)& \text{for }x\in (0,1).
\end{array}\right.
\end{equation}
Given any $\delta>0$, there is an operator $\mathcal{K}_1:L^2(0,1)\to \mathbb{R}$ depending on $\delta$, so that for each $y_0\in L^2(0,1)$, there exist constants ${k}_1=k_1\left(\delta,\|y_0\|_{L^2(0,1)}\right),{T}_1=T_1\left(\delta,\|y_0\|_{L^2(0,1)}\right)>1$ and $\hat{L}=\hat{L}\left(\delta,\|y_0\|_{L^2(0,1)}\right)\ge 1$, such that the system \eqref{Burgers-eq2} with 
\begin{equation}
    \label{u}
    {u}(t)=\int_{\min\{{T}_1,t\}}^t e^{k_1(t-s)}{\mathcal{K}}_1(y(s))ds
\end{equation}
admits a unique solution $y\in C([0,+\infty);L^2(0,1))$ satisfying
\begin{equation}
    \label{global-burgers}
    \|y(t)\|_{H^1(0,1)}\le\hat{L}e^{-\delta t},\;\;\forall t\ge 1.
\end{equation}
\end{theorem}
\begin{proof}
    From {Lemma} \ref{diss-property}, letting $\tau=\frac{1}{2}$, we have
    \begin{equation*}
         \| y(t)\|_{H^1(0,1)}\le L\left(\|y_0\|_{L^2(0,1)},\frac{1}{2}\right)e^{-\frac{c_1 }{2}t},\;\;\forall t>\frac{1}{2}.
    \end{equation*}
    Thus, there exists ${T}_1={T}_1\left(\delta,\|y_0\|_{L^2(0,1)}\right)\ge1$ such that
    \begin{equation*}
        \|y(T_1)\|_{H^1(0,1)}\le \rho,
    \end{equation*}
where \(\rho\) is defined in \eqref{rho}. For $t\in (0,{T}_1)$, no control is applied; for $t\in [{T}_1,+\infty)$, we adopt the same control constructed in {Theorem} \ref{main-th1}. In other words, the control takes the form as \eqref{u}. The well-posedness is obvious. Using the exponential stability \eqref{exp-s}, we have
\begin{equation*}
    \|y(t)\|_{H^1(0,1)}\le {{M}}e^{-\delta(t-{T}_1)}\|y({T}_1)\|_{H^1(0,1)}\le \rho {{M}}e^{\delta {T}_1}e^{-\delta t},\;\; \forall t\ge {T}_1.
\end{equation*}
Consider $t\in (1,{T}_1)$, we aim to choose $\hat{L}$ such that
\begin{equation*}
    Le^{-\frac{c_1}{2}t}\le \hat{L}e^{-\delta t},\;\;\forall 1<t<{T}_1,
\end{equation*}
which is equivalent to
\begin{equation*}
    \hat{L}\ge Le^{(\delta-\frac{c_1}{2})t},\;\;\forall 1<t<{T}_1.
\end{equation*}
Consequently, letting $\hat{L}\left(\delta,\|y_0\|_{L^2(0,1)}\right)=\max\left\{L,Le^{(\delta-\frac{c_1}{2}){T}_1},\rho {{M}}e^{\delta {T}_1}\right\},$ we finish the proof.
\end{proof}

\subsection{Allen-Cahn equation}

We next turn to the Allen-Cahn equation
\begin{equation}
    \label{C-Ieqs}
    \left\{\begin{array}{ll}
        \partial_t y(t,x)-\partial_{xx}y(t,x)-\kappa y(t,x) +\lambda y^3(t,x) =0&\quad \text{for }(t,x)\in (0,T)\times (0,1),\\
        y(t,0)=y(t,1)=0 &\quad \text{for }t\in (0,T),\\
        y(0,x)=y_0(x)&\quad \text{for }x\in (0,1).
\end{array}\right.
\end{equation}
Here $\lambda>0$ and $0\le \kappa< c_1 $. This is a well-posed semilinear heat equation. More precisely, given $y_0\in L^2(0,1)$, there exists a unique solution $y\in C([0,{T}];L^2(0,1))\cap L^2(0,{T};H_0^1(0,1))$. 

\begin{lemma}
    \label{thm1.1-to-ac}
     The nonlinear term $f(y)=-\lambda y^3$ with $\lambda>0$ satisfies the assumption (H).
\end{lemma}
\begin{proof}
It is readily verified that for every $y_1,y_2\in B_{{H^1}}(y_0,r)$, we have
    \begin{equation*}
             \begin{aligned}
                 \|f(y_1)-f(y_2)\|_{L^2(0,1)} &\le|\lambda| \|(y_1-y_2)^3\|_{L^2(0,1)}+3|\lambda|\|y_1^2y_2-y_2^2y_1\|_{L^2(0,1)}\\
            &\le |\lambda|\|y_1-y_2\|^3_{L^6(0,1)}+3|\lambda|\|y_1-y_2\|_{L^6(0,1)}\left(\int_0^1 (y_1(x)y_2(x))^3 dx\right)^{\frac{1}{3}}\\
            &\le |\lambda| \tilde{c}_2^3\|y_1-y_2\|^3_{H^1(0,1)}+3|\lambda| \tilde{c}_2^3\|y_1-y_2\|_{H^1(0,1)}(\|y_1\|^2_{H^1(0,1)}+\|y_2\|^2_{H^1(0,1)})\\
            &\le 20\tilde{c}_2^3|\lambda|(r^2+\|y_0\|^2_{H^1(0,1)})\|y_1-y_2\|_{H^1(0,1)}.
             \end{aligned}
         \end{equation*}
       Furthermore, it holds that
       \begin{equation*}
                \|f(y)\|_{L^2(0,1)}= \| \lambda y^3\|_{L^2(0,1)}\le |\lambda|\|y\|^2_{L^{\infty}(0,1)}\|y\|_{H^1(0,1)},\; \forall y\in H^1(0,1).
         \end{equation*}
        Hence, (H) is satisfied with $C(\|y_0\|_{H^1(0,1)},r)=20\tilde{c}_2^3|\lambda|(r^2+\|y_0\|^2_{H^1(0,1)})$ and $c(x)=|\lambda|x^2$.
\end{proof}

The following lemma reveals the dissipativity of the Allen-Cahn equation, and we only present the main idea of its proof in {Appendix}.

\begin{lemma}\label{ac-diss}
For each $\tau \in (0,T)$ and $y_0\in L^2(0,1)$, there is a constant $Z=Z\left(\|y_0\|_{L^2(0,1)},\tau\right)$ so that the solution of \eqref{C-Ieqs} satisfies
     \begin{equation}
     \label{ci-last}
         \|y(t)\|_{H^1(0,1)}\le Z\left(\|y_0\|_{L^2(0,1)},\tau\right)e^{-\frac{c_1 -\kappa}{2}t},\;\;\forall t>\tau.
     \end{equation}
\end{lemma}

Analogous to {Theorem \ref{global-exp-sta-burgers}}, we have the following result without proof.
\begin{theorem}
    Consider the Allen-Cahn equation with a boundary control
    \begin{equation}
    \label{C-I-eq2}
\left\{\begin{array}{ll}
\partial_t y(t,x)-\partial_{xx}y(t,x)-\kappa y(t,x) +\lambda y^3 (t,x)=0 & \text{for }(t,x)\in \mathbb{R}^+\times (0,1),\\[2mm]
y( t,0)=u(t)\quad\text{and}\quad y( t,1)=0& \text{for }t\in \mathbb{R}^+,\\[2mm]
y(t_0,x)=y_0(x)& \text{for }x\in (0,1),
\end{array}\right.
\end{equation}
where $\lambda>0$ and $0\le \kappa< c_1 $. Given any $\delta>0$, there is an operator $\mathcal{K}_2:L^2(0,1)\to \mathbb{R}$ depending on $\delta$, so that for each $y_0\in L^2(0,1)$, there exist constants ${k}_2=k_2\left(\delta,\|y_0\|_{L^2(0,1)}\right),{T}_2=T_2\left(\delta,\|y_0\|_{L^2(0,1)}\right)>1$ and $\hat{Z}=\hat{Z}\left(\delta,\|y_0\|_{L^2(0,1)}\right)\ge 1$, such that the system \eqref{C-I-eq2} with 
$${u}(t)=\int_{\min\{{T}_2,t\}}^t e^{k_2(t-s)}{\mathcal{K}}_2(y(s))ds$$ 
admits a unique solution $y\in C([0,+\infty);L^2(0,1))$ satisfying
\begin{equation*}
    \|y(t)\|_{H^1(0,1)}\le \hat{Z}e^{-\delta t}, \;\;\forall t\ge 1.
\end{equation*}
\end{theorem}

\begin{appendices}

  \renewcommand{\thesection}{}
  \titleformat{\section}[block]{\Large\bfseries}{\thesection}{0em}{}
  \renewcommand{\theequation}{A.\arabic{equation}}
\section{Appendix: Proofs of two lemmas}
The proofs can be conducted similarly through the multiplier technique presented in \cite[Section 3.1]{S}.

\begin{1proofof}{Lemma \ref{energy-estimatet}}

\noindent     (i) \textbf{$L^2$-norm of $y(t)$.  } Taking the inner product of the first equality of \eqref{burgers-eq} with $2 y(t)$, this yields
\begin{equation}
\label{L2}
    \partial_t\left(\left\|y(t)\right\|^2_{L^2(0,1)}\right)+2 \left\|\partial_x y(t)\right\|_{L^2(0,1)}^2=0,\;\;\forall t>0.
\end{equation}
Using the Poincar\'{e} inequality \eqref{po}, we derive
$$
\partial_t\left(\|y(t)\|^2_{L^2(0,1)}\right)+2 c_1 \|y(t)\|_{L^2(0,1)}^2 \le 0,\;\;\forall t>0,
$$
which implies that
$$\quad\|y(t)\|^2_{L^2(0,1)} \leq e^{-2 c_1  t}\left\|y_0\right\|_{L^2(0,1)}^2,\;\;\forall t>0.$$
\textbf{Mean $L^2$-norm of $\partial_x y(t)$.  } Integrating \eqref{L2} with respect to $t$, we have
$$
\int_{0}^t\|\partial_x y(s)\|^2_{L^2(0,1)} d s \le \frac{\|y_0\|^2_{L^2(0,1)}}{2 },\;\;\forall t>0.
$$
\textbf{$H^1$-norm of $y(t)$.   } Let $\tilde{t}\in (0,t)$. Taking the inner product of the first equality of \eqref{burgers-eq} with $-2(t-\tilde{t}) \partial_{xx} y(t)$, it is not difficult to check that
\begin{equation*}
\label{L2-H1}
\begin{aligned}
\partial_t\left((t-\tilde{t})\| \partial_x y(t) \|_{L^2(0,1)}^2\right) &- \|\partial_x y(t)\|_{L^2(0,1)}^2+2(t-\tilde{t}) \| \partial_{x x} y(t) \|_{L^2(0,1)}^2 = 2(t-\tilde{t})\left\langle y(t) \partial_x y(t), \partial_{x x} y(t)\right\rangle.
\end{aligned}
\end{equation*}
Using the inequality \eqref{Agmon1}, we bound the nonlinear term by
\begin{equation}
\label{nonlinear-product1}
    \left\langle y(t) \partial_x y(t), \partial_{x x} y(t)\right\rangle \le \frac{c_2^2}{4}\|y(t)\|_{L^2(0,1)}^2\|\partial_x y(t)\|_{L^2(0,1)}^4+ \frac{c_2^2}{4}\|\partial_x y(t)\|_{L^2(0,1)}^2+\frac{1}{2}\|\partial_{xx}y(t)\|_{L^2(0,1)}^2,
\end{equation}
it leads to
$$
\begin{aligned}
(s-\tilde{t})\|\partial_x y(s)\|_{L^2(0,1)}^2 &\le (1+\frac{c_2^2}{2 }t)\frac{\|y_0\|_{L^2(0,1)}^2}{2}\cdot e^{\frac{c_2^2}{4 }\|y_0\|_{L^2(0,1)}^4},\;\;\forall \tilde{t}<s\le t.
\end{aligned}
$$
Letting $s=t$ and $\tilde{t} \rightarrow 0$, we draw that 
$$
\left\|y(t)\right\|^2_{H^1(0,1)} \le 2\|y_0\|^2_{L^2(0,1)} +\frac{(1+\frac{c_2^2}{2 } t) \cdot\left\|y_0\right\|_{L^2(0,1)}^2 \cdot e^{\frac{c_2^2}{4 }\left\|y_0\right\|_{L^2(0,1)}^4}}{t}:=C_1\left(\|y_0\|_{L^2(0,1)},t\right),\;\;\forall t>0.
$$

\noindent (ii) \textbf{$H^1$-norm of $y(t)$.   } Taking the inner product of \eqref{burgers-eq} with $- \partial_{x x} y(t)$, this yields
\begin{equation}
    \label{H1}
\begin{aligned}
\frac{1}{2} \partial_t\left(\|\partial_x y(t))\|_{L^2(0,1)}^2\right)&+ \| \partial_{x x} y(t)\|_{L^2(0,1)}^2\le \frac{c_2^2}{4 }\left\|y(t)\right\|_{L^2(0,1)}^2\left\|\partial_x y{(t)}\right\|_{L^2(0,1)}^4\\
&\quad\quad\quad\quad\quad\quad+\frac{c_2^2}{4 }\left\|\partial_x y(t)\right\|_{L^2(0,1)}^2+\frac{1}{2}\|\partial_{xx}y(t)\|^2.
\end{aligned}
\end{equation}
It follows that
$$
\begin{aligned}
\left\| y(t)\right\|^2_{H^1(0,1)} &\le 2\|y(\tau)\|_{L^2(0,1)}^2+ (2+\frac{c_2^2}{2 })\|y(\tau)\|_{H^1(0,1)}^2  e^{\frac{c_2^2}{4 }\|y(\tau)\|_{H^1(0,1)}^4}:=C_2\left(\|y(\tau)\|_{H^1(0,1)}\right),\;\;\forall t>\tau.
\end{aligned}
$$
\textbf{Mean $L^2$-norm of $\partial_{xx} y(t).$   } Integrating \eqref{H1} with respect to $t$, we derive
\begin{equation*}
\begin{aligned}
    \int_{\tau}^t\|\partial_{xx} y(s)\|_{L^2(0,1)}^2ds&\le (1+\frac{c_2^2}{4 })\left\|y(\tau)\right\|_{H^1(0,1)}^2+\frac{c_2^2C_2}{4 }\left\|y(\tau)\right\|_{H^1(0,1)}^4:=M_1\left(\left\|y(\tau)\right\|_{H^1(0,1)}\right),\;\;\forall t>\tau.
    \end{aligned}
\end{equation*}
\textbf{Mean $L^2$-norm of $\partial_{t} y(t).$   } Since $\partial_t y(t,x)- \partial_{xx}y(t,x)+ y(t,x)\partial_x y(t,x)=0$, we have
\begin{equation*}
\begin{aligned}
    \int_{\tau}^t \|\partial_s y(s)\|^2_{L^2(0,1)} d s&\le 2 \int_{\tau}^t \|\partial_{xx} y(s)\|^2_{L^2(0,1)} d s+2 \int_{\tau}^t \|y(s)\partial_x y(s)\|^2_{L^2(0,1)} d s\\
    &\le 2M_1+\frac{c
    _2^2}{2}\|y(\tau)\|_{H^1(0,1)}^2(C_2\|y(\tau)\|_{H^1(0,1)}^2+1)\\
&:=M_2\left(\left\|y(\tau)\right\|_{H^1(0,1)}\right),\;\;\;\;\;\;\;\;\;\;\;\;\;\;\forall t>\tau.
    \end{aligned}
\end{equation*}
\textbf{$L^2$-norm of $\partial_t y(t).$   } Define $h(t, x):=\partial_t y(t, x)$. Then $h$ satisfies that
\begin{equation}
    \label{hweak}
    \partial_t h(t, x)- \partial_{x x} h(t, x)+ h(t, x) \partial_x y{(t, x)}+ y{(t, x)} \partial_x h(t, x)=0.
\end{equation}
It is not hard to check that
\begin{equation*}
\label{yt-L2}
\partial_t\left(\left(t-\tau\right)\|h(t)\|_{L^2(0,1)}^2\right)-\|h(t)\|_{L^2(0,1)}^2+2\left(t-\tau\right)  \| \partial_x h(t) \|_{L^2(0,1)}^2=2(t-\tau) \langle y{(t)} \partial_x h(t), h(t)\rangle.
\end{equation*}
Note that
\begin{equation}
    \label{nonlinear-product2}
    \left\langle y(t) \partial_x h(t), h{(t)}\right\rangle \le \frac{c^2_2 }{2}\|y(t)\|_{L^2(0,1)}\|\partial_x y(t)\|_{L^2(0,1)}\| h(t)\|_{L^2(0,1)}^2+\frac{1}{2}\| \partial_x h(t)\|_{L^2(0,1)}^2,
\end{equation}
which leads to
\begin{equation*}
    \|h(t)\|_{L^2(0,1)}^2\le \frac{M_2\cdot e^{(\frac{c_2^2}{4c_1}+\frac{c_2^2}{4 })\|y(\tau)\|^2_{H^1(0,1)}}}{t-\tau}:=M_3\left(\|y(\tau)\|_{H^1(0,1)},t,\tau\right),\;\;\forall t>\tau.
\end{equation*}
\textbf{$H^2$-norm of $y(t).$   } From the identity $\partial_t y(t,x)- \partial_{xx}y(t,x)+y(t,x)\partial_x y(t,x)=0$, we have
\begin{equation*}
\begin{aligned}
    \|\partial_{xx} y(t)\|_{L^2(0,1)}^2 &\le 2  \|\partial_{t} y(t)\|_{L^2(0,1)}^2 +2  \|y(t)\partial_x y(t)\|_{L^2(0,1)}^2\\
    &\le 2M_3+2c
    _2^2 C^3_2\|y(\tau)\|_{H^1(0,1)}.
    \end{aligned}
\end{equation*}
Then it holds that
$$\|y(t)\|^2_{H^2(0,1)}\le 2C_2+4M_3+2c
    _2^2 C^3_2\|y(\tau)\|_{H^1(0,1)}:=C_3\left(\|y(\tau)\|_{H^1(0,1)},t,\tau\right),\;\;\forall t>\tau.$$

 \noindent   (iii) \textbf{$L^2$-norm of $\partial_t y(t).$   } Taking the inner product of \eqref{hweak} wish $h(t)$, we derive
\begin{equation*}
    \label{yt-h1}
\partial_t\left(\frac{1}{2}\|h(t)\|_{L^2(0,1)}^2\right)+ \left\|\partial_x h(t)\right\|^2_{L^2(0,1)}  =\left\langle y(t) \partial_x h(t), h{(t)}\right\rangle.
\end{equation*}
From \eqref{nonlinear-product2}, it is readily verified that
$$
\partial_t\left(\frac{1}{2}\|h(t)\|_{L^2(0,1)}^2\right) \le \frac{c_2^2 }{2 }\|y{(t)}\|_{L^2(0,1)}\|\partial_x y(t)\|_{L^2(0,1)}\|h(t)\|_{L^2(0,1)}^2,
$$
and thus
$$
\left\|h(t)\right\|_{L^2(0,1)}^2\le \left\|h(\tau)\right\|_{L^2(0,1)}^2 e^{(\frac{c_2^2}{4c_1}+\frac{c_2^2}{4 })\|y(\tau)\|^2_{H^1(0,1)}},\;\;\forall t>\tau.
$$
Note that $\left\|h(\tau)\right\|_{L^2(0,1)}^2 \le 2\left\|\partial_{xx} y(\tau)\right\|_{L^2(0,1)}^2 +2\left\|y(\tau)\partial_x y(\tau)\right\|_{L^2(0,1)}^2 \lesssim \left\|y(\tau)\right\|_{H^2(0,1)}^2+\left\|y(\tau)\right\|_{H^2(0,1)}^4$. Hence, it follows that
$$
\left\|h(t)\right\|^2_{L^2(0,1)}\le C(\left\|y(\tau)\right\|_{H^2(0,1)}^2+\left\|y(\tau)\right\|_{H^2(0,1)}^4)e^{(\frac{c_2^2}{4c_1{ }^2}+\frac{c_2^2}{4 })\|y(\tau)\|_{H^1(0,1)}}:=M_4\left(\|y(\tau)\|_{H^2(0,1)}\right),\;\;\forall t>\tau.
$$
   \textbf{$H^2$-norm of $y(t).$   } Since $\partial_t y(t,x)-  \partial_{xx}y(t,x)+y(t,x)\partial_x y(t,x)=0$, we draw that
\begin{equation*}
\begin{aligned}
    \|\partial_{xx} y(t)\|_{L^2(0,1)}^2 &\le 2 \|\partial_{t} y(t)\|_{L^2(0,1)}^2 +2  \|y(t)\partial_x y(t)\|_{L^2(0,1)}^2\le 2 M_4+2c
    _2^2 C^3_2\|y(\tau)\|_{H^1(0,1)}.
    \end{aligned}
\end{equation*}
Therefore, we derive
$$\|y(t)\|^2_{H^2(0,1)}\le 2C_2+4 M_4+2c
    _2^2 C^3_2\|y(\tau)\|_{H^1(0,1)}:=C_4\left(\|y(\tau)\|_{H^2(0,1)}\right),\;\;\forall t>\tau.$$
    
    \noindent The proof is completed.
\end{1proofof}
\begin{2proofof}{Lemma \ref{ac-diss}}
    The proof is similar to the case of the Burgers equation, except the estimates for $\left\langle y^3(t), \partial_{x x} y(t)\right\rangle$ and $\langle y^2{(t)} h(t), h(t)\rangle$. To begin with, taking the inner product of the first equality of \eqref{C-Ieqs} with $2 y(t)$, we have
\begin{equation}
\label{L2-C-I}
    \partial_t\left(\|y(t)\|_{L^2(0,1)}^2\right)+2 \|\partial_x y(t)\|_{L^2(0,1)}^2+2\lambda\|y^2(t)\|_{L^2(0,1)}^2=2\kappa\|y(t)\|_{L^2(0,1)}^2. 
\end{equation}
Using the inequality \eqref{Agmon1}, we obtain that
$$
\partial_t\left(\|y(t)\|_{L^2(0,1)}^2\right)+2 (c_1 -\kappa) \|y(t)\|_{L^2(0,1)}^2 \le 0,\;\;\forall t>0,
$$
and thus
$$\quad\|y(t)\|_{L^2(0,1)}^2 \leq e^{-2 (c_1 -\kappa) t}\left\|y_0\right\|_{L^2(0,1)}^2\le \left\|y_0\right\|_{L^2(0,1)}^2,\;\;\forall t>0.$$ 
 Analogous to \eqref{nonlinear-product1} and \eqref{nonlinear-product2}, we have
$$
\begin{aligned}
\left\langle y^3(t), \partial_{x x} y(t)\right\rangle &\le\|y^3(t)\|_{L^2(0,1)}\|\partial_{x x} y(t)\|_{L^2(0,1)} \\
&\le \|y(t)\|^3_{L^{\infty}(0,1)}\|\partial_{x x} y(t)\|_{L^2(0,1)}\\
&\le c^3_2\|y(t)\|_{L^2(0,1)}^{\frac{3}{2}}\|\partial_x y(t)\|_{L^2(0,1)}^{\frac{3}{2}}\|\partial_{x x} y(t)\|_{L^2(0,1)}\\
&\le \frac{c^6_2}{2}\left\|y(t)\right\|^3_{L^2(0,1)}\left\|\partial_x y(t)\right\|_{L^2(0,1)}^3+\frac{1}{2}\left\|\partial_{x x} y(t)\right\|_{L^2(0,1)}^2\\
&\le \frac{c_2^6}{4}\|y(t)\|_{L^2(0,1)}^2\|\partial_x y(t)\|_{L^2(0,1)}^4+ \frac{c_2^6}{4}\|y(t)\|_{L^2(0,1)}^4\|\partial_x y(t)\|_{L^2(0,1)}^2+\frac{1}{2}\|\partial_{xx}y(t)\|_{L^2(0,1)}^2\\
&\le \frac{c_2^6}{4}\|y(t)\|_{L^2(0,1)}^2\|\partial_x y(t)\|_{L^2(0,1)}^4+ \frac{c_2^6}{4}\|y_0\|_{L^2(0,1)}^4\|\partial_x y(t)\|_{L^2(0,1)}^2+\frac{1}{2}\|\partial_{xx}y(t)\|_{L^2(0,1)}^2,
\end{aligned}
$$
and
$$
\begin{aligned}
\langle y^2{(t)}  h(t), h(t)\rangle  &\le \|y^2{(t)}h(t)\|_{L^2(0,1)}\| h(t)\|_{L^2(0,1)}\\
&\le \|y{(t)}\|^2_{L^{\infty}(0,1)}\| h(t)\|_{L^2(0,1)}^2\\
&\le c_2^2\|y(t)\|_{L^2(0,1)}\|\partial_x y(t)\|_{L^2(0,1)}\| h(t)\|_{L^2(0,1)}^2.
\end{aligned}
$$
Hence, with reference to {Lemma \ref{energy-estimatet}} and {Lemma \ref{diss-property}}, similar energy estimates and dissipativity can be derived. We omit the details of the remaining part.
\end{2proofof}
\end{appendices}

\end{document}